\documentclass[11pt]{article}
\usepackage[english,activeacute]{babel}
\usepackage{amsmath,amsfonts,amssymb,amstext,amsthm,amscd,mathrsfs,amsbsy,xypic,centernot,multicol,xcolor}
\usepackage{graphicx}
\usepackage{hyperref}

\newtheorem{teo}{Theorem}[section]
\newtheorem{pro}[teo]{Proposition}
\newtheorem{coro}[teo]{Corollary}
\newtheorem{lem}[teo]{Lemma}

\theoremstyle{definition}
\newtheorem{defi}[teo]{Definition}
\newtheorem{exam}[teo]{Example}
\newtheorem{rem}[teo]{Remark}
\newtheorem{nota}[teo]{Notation}

\newcommand{\N}{\mathbb N}
\newcommand{\Z}{\mathbb Z}

\newcommand{\R}{\mathbb R}
\newcommand{\C}{\mathbb C}

\newcommand{\m}{\mathfrak{m}}

\newcommand{\be}{\beta}
\newcommand{\ga}{\gamma}

\DeclareMathOperator{\spn}{span}

\begin{document}

\title{Factorization of the normalization of the Nash blow-up of order $n$ of $\mathcal{A}_{n}$ by the minimal resolution}
\author{Enrique Ch\'avez-Mart\'inez}
\maketitle

\begin{abstract}
We show that the normalization of the Nash blow-up of order $n$ of the toric surface singularity $\mathcal{A}_{n}$ can be factorized by the minimal resolution of $\mathcal{A}_{n}$. The result is obtained using the combinatorial description of these objects.


\end{abstract}



\section*{Introduction}

The Nash blowup of an algebraic variety is a modification that replaces singular points by limits of tangent spaces at non-singular points. It was proposed to solve singularities by iterating this process \cite{No, S}. This question has been treated in \cite{No,R,GS-1,GS-2,Hi,Sp,At, T, dFDo}. The particular case of toric varieties is treated in \cite{GS-1, GT, GM, D1, DG} using their combinatorial structure.

There is a generalization of Nash blowups, called higher Nash blowups or Nash blowups of order $n$, that was proposed by Takehiko Yasuda. This modification replaces singular points  by limits of infinitesimal neighborhoods of certain order at non-singular points. In particular, the higher Nash blowup looks for resolution of singularities in one step \cite{Y1}. T. Yasuda proves that this is true for curves in characteristic zero, but conjectures that is false in general, proposing as a counterexample the toric surface $\mathcal{A}_{3}$. 


There are several papers that deal with higher Nash blowups in the special case of toric varieties. The usual strategy for this special case is to translate the original geometric problem into a combinatorial one and then try to solve the latter. So far, the combinatorial description of higher Nash blowups of toric varieties has been obtained using Gr\"{o}ebner fans or higher-order Jacobian matrices.

The usage of Gr\"{o}ebner fans for higher Nash blowups of toric varieties was initiated in \cite{D2}. Later, this tool was further developed in \cite{T} to show that the Nash blowup of order $n$ of the toric surface singularity $\mathcal{A}_{3}$ is singular for any $n > 0$, over the complex numbers. This problem was later revisited to show that it is also holds in prime characteristic \cite{DN}. 


The techniques from \cite{T} can be used to compute the Gr\"{o}ebner fan of the normalization of higher Nash blowup of $\mathcal{A}_{n}$ for some $n$'s. Those computations suggest that the essential divisors of the minimal resolution of $\mathcal{A}_{n}$ appear in the normalization of the Nash blow-up of order $n$ of $\mathcal{A}_{n}$ for some $n$'s. The main goal of this paper is to show that this happens for all $n$. In particular, this implies that the normalization of the Nash blowup of order $n$ of $\mathcal{A}_{n}$ factors through its minimal resolution (see Corollary \ref{ct}). 

The approach to study higher Nash blowups of toric varieties using a higher order Jacobian matrix was initiated in \cite{ChDG}. That paper deals with a conjecture proposed by T. Yasuda  concerning the semigroup associated to the higher Nash blowup of formal curves. There it was proved that the conjecture is true in the toric case but false in general. This was achieved by studying properties of the higher order Jacobian matrix of monomial morphisms. In this paper we follow a similar but more general approach. 


The normalization of the higher Nash blowup of $\mathcal{A}_{n}$ is a toric variety associated to a fan that subdivides the cone determining $\mathcal{A}_{n}$ \cite{ChDG, GT}. An explicit description of this fan could be obtained by effectively computing all minors of the corresponding higher order Jacobian matrix. This is a difficult task given the complexity of the matrix for large $n$. However, for the problem we are interested in, we do not require an explicit description of the entire fan.

The rays that subdivides the cone of $\mathcal{A}_{n}$ to obtain its minimal resolution can be explicitly specified. Thus, in order to show that these rays appear in the fan associated to the normalization of the higher Nash blowup we need to be able to control only certain minors of the matrix. A great deal of this paper is devoted to construct combinatorial tools that allow us to accomplish that goal.

\begin{section}{The main result}
In this section we state the main result of this work. First, we introduce some notation that will be constantly used throughout this paper. \textit{From now on, $n$ will always denote a fixed positive natural number}.

\begin{nota}\label{notacion}Let $\ga,\be\in\N^{t}$ and $v\in\N^{2}$.
\begin{itemize}
    \item[1)] We denote $\pi_{i}(\be)$  the projection to the $i$-th  coordinate of $\be$.
    \item[2)] $\ga\leq\be$ if and only if $\pi_{i}(\ga)\leq\pi_{i}(\be)$ for all $i\in\{1,\ldots,t\}$. In particular, $\ga<\be$ if and only if $\ga\leq\be$ and $\pi_{i}(\ga)<\pi_{i}(\be)$ for some $i\in\{1,\ldots,t\}$.
    \item[3)] $\binom{\beta}{\ga}:=\prod_{i=1}^{t}\binom{\pi_{i}(\be)}{\pi_{i}(\ga)}.$
    
    \item[4)]$|\beta|=\sum_{i=1}^{t}\pi_{i}(\be)$.
    \item[5)]$\Lambda_{t,n}:=\{\be\in\N^{t}\mid 1\leq|\be|\leq n\}$. In addition, $\lambda_{t,n}:=|\Lambda_{t,n}|=\binom{n+t}{n}-1$.
    \item[6)]$\bar{v}:=\Big{(}\binom{v}{\alpha}\Big{)}_{\alpha\in\Lambda_{2,n}}\in\N^{\lambda_{2,n}}.$
    \item[7)]Let $A_{n}:=\begin{pmatrix}
1 & 1 & n\\
0 & 1 & n+1
\end{pmatrix}.$
    \item[8)]Given $J\subset\Lambda_{3,n}$, let $m_{J}:=\sum_{\be\in J}A_{n}\be\in\N^{2}.$

\end{itemize}
\end{nota}

Let $X\subset\C^{s}$ be an irreducible algebraic variety of dimension $d$. For a non-singular point $x\in X$, the $\C$-vector space $(\m_{x}/\m_{x}^{n+1})^{\vee}$ has dimension $\lambda_{d,n}$, where $\m_{x}$ denotes the maximal ideal of $x$.

\begin{defi}
{\cite{No, OZ, Y1}} With the previous notation, consider the morphism of Gauss: \begin{align}
G_n:X\setminus Sing(X&)\rightarrow Gr(\lambda_{d,n},\C^{\lambda_{s,n}})\notag\\
x&\mapsto (\m_{x}/\m_{x}^{n+1})^{\vee},\notag
\end{align}
where $Sing(X)$ denotes the singular locus of $X$ and $Gr(\lambda_{d,s},\C^{\lambda_{s,n}})$ is the Grassmanian of vector subspaces of dimension $\lambda_{d,n}$ in $\C^{\lambda_{s,n}}$. 

    Denote by $Nash_{n}(X)$ the Zariski closure of the graph of $G_{n}$. Call $\pi_{n}$ the restriction to $Nash_{n}(X)$ of the projection of $X\times Gr(\lambda_{d,n},\C^{\lambda_{s,n}})$ to $X$. The pair $(Nash_{n}(X),\pi_{n})$ is called the Nash blow-up of $X$ of order $n$.
\end{defi}

This entire paper is devoted to study some aspects of the higher Nash blowup of the $\mathcal{A}_n$ singularity. Let us recall its definition and the notation we will use. 

\begin{defi}\label{an} Consider the cone $\sigma_{n}=\R_{\geq0}\{(0,1),(n+1,-n)\}\subset(\R^{2})^{\vee}$. 
We denote as $\mathcal{A}_{n}$ the normal toric surface corresponding to $\sigma_{n}$, i.e., $\mathcal{A}_{n}=V(xz-y^{n+1})$.
\end{defi}


In \cite{ChDG}, the higher Nash blowup is studied through a higher-order Jacobian matrix. It is worth mentioning that there are other versions of higher order Jacobian matrices \cite{D3, BD, HJN}. In the context of toric varieties, that matrix gave place to the following definition. 

\begin{defi}{\cite[Proposition 2.4]{ChDG}}
Let $J\subset\Lambda_{3,n}$ be such that $|J|=\lambda_{2,n}$. We define the matrix 
$$L_{J}^{c}:=\Big{(}c_{\be}\Big{)}_{\be\in J},$$
where $c_{\be}=\sum_{\ga\leq \be}(-1)^{|\be-\ga|}\binom{\be}{\ga}\overline{A_{n}\ga}\in\N^{\lambda_{2,n}}$. In addition, we denote $$S_{A_{n}}:=\{J\subset\Lambda_{3,n}\mid |J|=\lambda_{2,n}\,\,\,\mbox{and}\,\,\,\det L_{J}^{c}\neq 0\}.$$
\end{defi}

\begin{pro}{\cite[Proposition 3.15]{ChDG}}\label{propomin}
Let $I_{n}=\langle x^{m_{J}}\mid J\in S_{A_{n}}\rangle$. Then $Nash_{n}(\mathcal{A}_{n})\cong Bl_{I_{n}}(\mathcal{A}_{n})$, where $Bl_{I_{n}}(\mathcal{A}_{n})$ is the blow-up of $\mathcal{A}_{n}$ centered on $I_{n}$.

\end{pro}

 Abusing the notation, let $I_{n}=\{m_{J}\in\R^{2}\mid J\in S_{A_{n}}\}$. The set $I_{n}$ defines an order function:

\begin{align*}
    \mbox{ord}_{I_{n}}:\sigma_{n}\rightarrow & \R \\
      v\mapsto & \min_{m_{J}\in I_{n}}\langle v, m_{J}\rangle.
\end{align*}

This function induces the following cones 
{\color{red}(cita GT):}

$$\sigma_{m_{J}}:=\{v\in\sigma_{n}\mid \mbox{ord}_{I_{n}}(v)=\langle v,m_{J}\rangle\}.$$
These cones form a fan $\Sigma(I_{n}):=\bigcup_{m_{J}\in I_{n}}\sigma_{m_{J}}.$ This fan is a refinement of $\sigma$.

\begin{pro}\label{propoblow}
With the previous notation, we have: $$\overline{Nash_{n}(\mathcal{A}_{n})}\cong X_{\Sigma(I_{n})},$$ where $\overline{Nash_{n}(\mathcal{A}_{n})}$ is the normalization of the Nash blow-up of $\mathcal{A}_{n}$ of order $n$ and $X_{\Sigma(I_{n})}$ is the normal variety corresponding to $\Sigma(I_{n})$.
\end{pro}
\begin{proof}By the previous proposition we have that $Nash_{n}(\mathcal{A}_{n})$ is a monomial blow-up. The result follows from proposition $5.1$ and remark $4.6$ of \cite{GT}.

\end{proof}

The goal of this paper is to prove the following result about the shape of the fan $\Sigma(I_{n})$.
\begin{teo}\label{mainn}
For each $k\in\{1,\ldots,n\}$, there exists $J,J'\in S_{\mathcal{A}_{n}}$ such that $(k,1-k)\in \sigma_{m_{J}}\cap\sigma_{m_{J'}}$. In particular, the rays generated by $(k,1-k)$ appear in the fan $\Sigma(I_{n})$.
\end{teo}

\begin{coro}\label{ct}Let $\mathcal{A}_{n}'$ be the minimal resolution of $\mathcal{A}_{n}$ and let $\overline{Nash_{n}(\mathcal{A}_{n})}$ be the normalization of the higher Nash blow-up of $\mathcal{A}_{n}$ of order $n$. Then there exists a proper birational morphism $\phi:\overline{Nash_{n}(\mathcal{A}_{n})}\rightarrow \mathcal{A}_{n}'$ such that the following diagram commutes
$$\xymatrix{ \overline{Nash_{n}(\mathcal{A}_{n})} \ar@{.>}[r]^{\phi}\ar[rd] & \mathcal{A}_{n}'\ar[d]\\
& \mathcal{A}_{n}.}$$
\end{coro}
\begin{proof}
It is well-known that $\mathcal{A}_{n}'$ is obtained by subdividing $\sigma_n$ with the rays generated by the vectors $(k,1-k)$, for $k\in\{1,\ldots,n\}$. The result follows by Theorem \ref{mainn}.


\end{proof}

\end{section}

\begin{section}{A particular basis for the vector space $\C^{\lambda_{2,n}}$}


As stated in Theorem \ref{mainn}, we need to find some subsets $J\subset\Lambda_{3,n}$ such that the determinant  of $L_{J}^{c}$ is non-zero. This will be achieved by reducing the matrix $L_{J}^{c}$ to another matrix given by vectors formed by certain binomial coefficients. In this section, we prove that those vectors are linearly independent. We will see that this is equivalent to finding some basis of the vector space $\C^{\lambda_{2,n}}$.




\begin{defi}\label{chain}Consider a sequence $\eta=(z,d_{0},d_{1},d_{2},\ldots,d_{r})$, where $z\in\Z_{2}$, $d_{0}=0$, $\{d_{i}\}_{i=1}^{r}\subset\N\setminus\{0\}$ and $\sum_{i=0}^{r}d_{i}=n$. We denote as $\Omega$ the set of all these possible sequences.
\end{defi}

With this set let us define a subset of vectors of $\N^{2}$.

    \begin{defi}\label{Teta} Let $\eta=(z,d_{0},d_{1},\ldots,d_{r})\in\Omega$. We construct a set of vectors $\{v_{j,\eta}\}_{j=1}^{n}\subset\N^{2}$ as follows.
    For each $j\in\{1,\ldots,n\}$, there exists an unique $t\in\{1,\ldots,r\}$ such that $\sum_{i=0}^{t-1}d_{i}<j\leq\sum_{i=0}^{t}d_{i}$. This implies that $j=\sum_{i=0}^{t-1}d_{i}+c$, where $0<c\leq d_{t}$. Then we define
    
    $$v_{j,\eta}= \left\{ \begin{array}{lcc}
             (\sum_{\substack{i\,\,\, \text{odd}\\ i<t}}d_{i}+c,0) &   if  & z=1\,\,\,\text{and}\,\,\,t\,\,\,\text{odd}, \\
             \\ (0,\sum_{\substack{i\,\,\, \text{even}\\ i<t}}d_{i}+c)  &  if &  z=1\,\,\,\text{and}\,\,\,t\,\,\,\text{even},\\
             \\ (0,\sum_{\substack{i\,\,\, \text{odd}\\ i<t}}d_{i}+c) &  if  & z=0\,\,\,\text{and}\,\,\,t\,\,\,\text{odd}, \\
             \\ (\sum_{\substack{i\,\,\, \text{even}\\ i<t}}d_{i}+c,0) & if 
             & z=0\,\,\,\text{and}\,\,\,t\,\,\,\text{even}. \\
             \end{array}
   \right.$$
   In addition, for each $j\in\{1,\ldots, n\}$, we denote $$T_{j,\eta}:=\{v_{j,\eta},v_{j,\eta}+(1,1),\ldots,v_{j,\eta}+(n-j)(1,1)\}.$$
   Furthermore, we denote $v_{0,\eta}:=(1,1)$ and
   $T_{0,\eta}:=\{(1,1),\ldots,(n,n)\}.$ We define $$T_{\eta}:=\bigcup_{j=0}^{n}T_{j,\eta}.$$
   Finally, recalling notation \ref{notacion}, we define 
$$\overline{T_{\eta}}=\{\bar{v}\in\C^{\lambda_{2,n}}\mid v\in T_{\eta}\}.$$
  \end{defi}
  
\begin{rem}Notice that this construction depends only on $\eta$. Moreover, geometrically, this construction is equivalent to taking vectors in an ordered way on the axes of $\N^2$.

\end{rem}    
    
\begin{exam}Let $n=6$, $r=5$ and $\eta=(1,0,1,1,1,1,2)$. For $j=3$ we have that 
$d_0+d_1+d_2<3= d_0+d_1+d_2+d_3$, then $t=3$, $v_{3,\eta}=(2,0)$ and $T_{3,\eta}=\{(2,0),(3,1),(4,2),(5,3)\}.$ $T_{\eta}$ is computed similarly and can be seen in the following figure.
 
\end{exam}    

\begin{figure}[h]\label{fig1}
\caption{Example of $T_{\eta}$ for $\eta=(1,0,1,1,1,1,2)$}
\centering
\includegraphics[scale=.16]{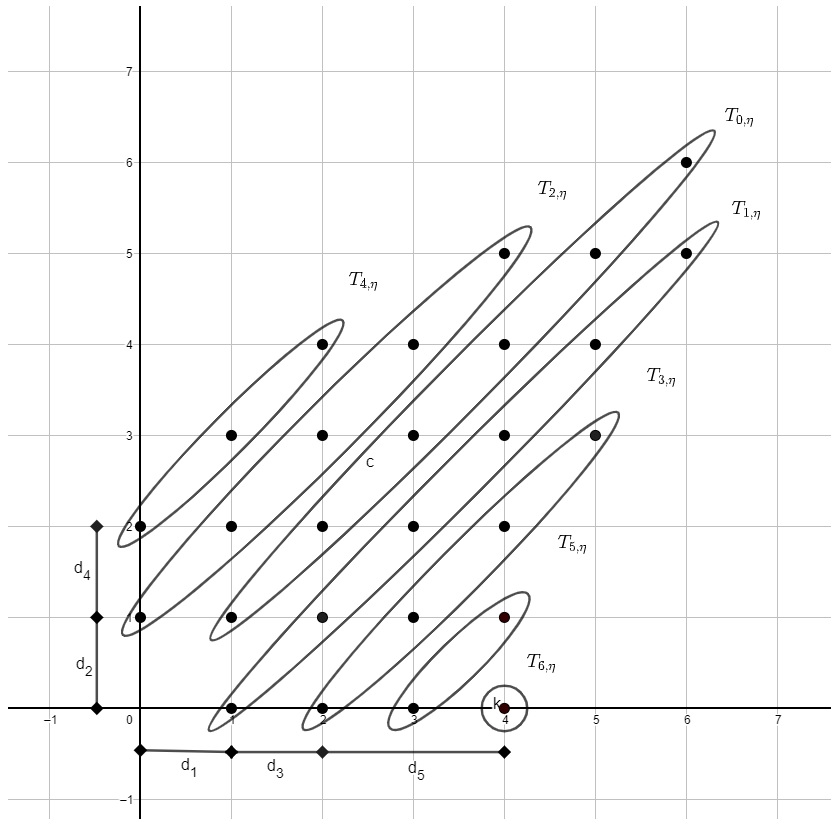}
\end{figure}
Now we give some basic properties of definition \ref{Teta}.


\begin{lem}\label{rem2}
Let $\eta\in\Omega$ and $u,v\in\N^2$. Then we have the following properties:
\begin{itemize}
    \item[1)] If $u\neq v$, then $\bar{u}\neq \bar{v}$.
    \item[2)] $|\overline{T_{\eta}}|=\lambda_{2,n}$.
    \item[3)] If $v_{j,\eta}=(l,0)$ or $v_{j,\eta}=(0,l)$, then $l\leq j$.
    \item[4)] $\pi_{i}(v)\leq n$ for all $v\in T_{\eta}$ and $i\in\{1,2\}$.
    \item[5)]If $v_{j,\eta}=(0,p)$, then for all $q< p$
there exists $l< j$ such that $v_{l,\eta}=(0,q)$. If $v_{j,\eta}=(p,0)$, then for all $q< p$
there exists $l<j$ such that $v_{l,\eta}=(q,0)$. 
    \item[6)]If $v_{j,\eta}=(0,l)$, then $\{v_{i,\eta}\}_{i=1}^{j}=\{(0,t)\}_{t=1}^{l}\cup\{(s,0)\}_{s=1}^{j-l}$. If $v_{j,\eta}=(l,0)$, then $\{v_{i,\eta}\}_{i=1}^{j}=\{(0,t)\}_{t=1}^{j-l}\cup\{(s,0)\}_{s=1}^{l}$.
\end{itemize}
\end{lem}

\begin{proof}
\begin{itemize}
    \item[1)]Since $u\neq v$, $\pi_{1}(u)\neq\pi_{1}(v)$ or $\pi_{2}(u)\neq\pi_{2}(v)$. Suppose the first case, the other is analogous. By definition $\bar{u}=\Big(\binom{u}{\alpha}\Big)_{\alpha\in\Lambda_{2,n}}$. Notice that $(1,0)\in\Lambda_{2,n}$. Then $$\bar{u}=\Big(\binom{u}{\alpha}\Big)=\Big(\pi_{1}(u),\ldots\Big)\neq \Big(\pi_{1}(v),\ldots\Big)=\Big(\binom{v}{\alpha}\Big)=\bar{v}.$$
    \item[2)]Notice that for each $j\in\{1,\ldots,n\}$, $|T_{j,\eta}|=n-j+1$ and $|T_{0,\eta}|=n$, this implies $|T_{\eta}|=(n+1)(n+2)/2-1=\lambda_{2,n}$. By the previous item we have that $|\overline{T_{\eta}}|=\lambda_{2,n}$.
    \item[3)] Let $t\leq r$ be such that $j=\sum_{i=0}^{t-1}d_{i}+c$. By definition $l=\sum_{\substack{i\,\,\, \text{odd}\\ i<t}}d_{i}+c$ or $l=\sum_{\substack{i\,\,\, \text{even}\\ i<t}}d_{i}+c$. In any case $l\leq j$.
    \item[4)] Let $v\in T_{\eta}$. If $v\in T_{0,\eta}$, then $v=(p,p)$, with $p\leq n$. If $v\notin T_{0,\eta}$, by definition \ref{Teta}, we have that $v=v_{j,\eta}+p(1,1)$, with $p\leq n-j$. Then $$\pi_{i}(v)=\pi_{i}(v_{j,\eta})+\pi_{i}(p(1,1))=\pi_{i}(v_{j,\eta})+p\leq j+p\leq n.$$
    \item[5)] Let $t\leq r$ be such that $j=\sum_{i=0}^{t-1}d_{i}+c$. Consider the case $v_{j,\eta}=(0,p)$. Suppose that $t$ is odd. By definition \ref{Teta}, $p=\sum_{\substack{i\,\,\, \text{odd}\\ i<t}}d_{i}+c$ and $z=0$. Let $q< p$. Then $q=\sum_{\substack{i\,\,\, \text{odd}\\ i<t'}} d_{i}+c'$, where $t'$ is odd, $t'< t$ and $c'\leq d_{t'}$ or $t'=t$ and $c'< c$. In any case, consider $l=\sum_{i=0}^{t'-1}d_{i}+c'$. Since $t'$ is odd and $z=0$,
    $v_{l,\eta}=(0,\sum_{\substack{i\,\,\, \text{odd}\\ i<t'}}d_{i}+c')=(0,q)$. If $t$ is even, we have that $p=\sum_{\substack{i\,\,\, \text{even}\\ i<t}}d_{i}+c$ and $z=1$. In this case the proof is identical. If $v_{j,\eta}=(p,0)$, the argument is analogous.
    \item[6)] If $v_{j,\eta}=(0,l)$, by the previous point, we have that $\{(0,t)\}_{t=1}^{l}\subset\{v_{i,\eta}\}_{i=1}^{j}$. On the other hand, we have that there exists $\{i_{1},\ldots,i_{j-l}\}$ such that $i_{p}\leq j$ and $v_{i_{p},\eta}\notin\{(0,t)\}_{t=1}^{j}$ for all $p\in\{1,\ldots,j-l\}$. Since $i_{p}<j$ for all $p$ and using the previous point, we obtain that $v_{i_{p},\eta}=(s_{p},0)$ for some $s_{p}\in\N$ and by the previous point $\{v_{i_{p},\eta}\}_{p=1}^{j-l}=\{(s,0)\}_{s=1}^{j-l}$. This implies that $\{v_{i,\eta}\}_{i=1}^{n}=\{(0,t)\}_{t=1}^{l}\cup\{(s,0)\}_{s=1}^{n-l}$.
\end{itemize}
\end{proof}

  

\begin{subsection}{Linear independence of $\overline{T_{\eta}}$}

By $2)$ of lemma \ref{rem2} we know that the cardinality of $\overline{T_{\eta}}$ is $\lambda_{2,n}$. In order to prove that it is a basis of $\C^{\lambda_{2,n}}$ we only have to see that it is linearly independent. For that we need some preliminary lemmas.
  
 
  \begin{lem}\label{lemacomb1}
  Let $0<c_{0}<c_{1}<\cdots<c_{l}$ be natural numbers. Then $\det\Big(\binom{c_{i}}{j}\Big)_{\substack{0\leq i\leq l\\ 0\leq j\leq l}}\neq 0$. In particular, the set of vectors 
   $\{\Big(\binom{c_{i}}{j}\Big)_{0\leq j\leq l}\in\C^{l+1}\mid 0\leq i\leq l\}$ is  linearly independent.
  \end{lem}
  \begin{proof}For each $j\leq l$, consider the polynomial $b_{j}(x)=\frac{x(x-1)\cdots(x-j+1)}{j!}$ and $b_{0}=1$. 
  Notice that for $x\in\N$, we have $b_{j}(x)=\binom{x}{j}$ and $\deg b_{j}(x)=j$ for all $j\in\{0,\ldots,l\}$. Thus,
  $$\Big(\binom{c_{i}}{j}\Big)_{\substack{0\leq i\leq l\\ 0\leq j\leq l}}=\Big(b_{j}(c_{i})\Big)_{\substack{0\leq i\leq l\\ 0\leq j\leq l}}.$$
  We show that the columns of this matrix are linearly independent. Let $\alpha_{0},\ldots,\alpha_{l}\in\C$ be such that $\sum_{j=0}^{l}\alpha_{j}b_{j}(c_{i})=0$ for each  $i\in\{0,\ldots,l\}$. Consider  $f(x)=\sum_{j=0}^{l}\alpha_{j}b_{j}(x)$. Then  $\{c_{0},\ldots,c_{l}\}$ are roots of $f(x)$. Since  $\deg f(x)\leq l$, we obtain that $f(x)=0$. Since $\deg b_{j}(x)=j$, we conclude $\alpha_{j}=0$ for all $j$.
  \end{proof}
  
  

  As we mentioned before, the goal is to prove that given $\eta\in\Omega$, the set of vectors $\overline{T_{\eta}}$ is linearly independent on $\C^{\lambda_{2,n}}$. Consider  \begin{equation}\label{equation5}
      \sum_{\bar{v}\in\overline{T_{\eta}}}a_{\bar{v}}\bar{v}=\bar{0}\in\C^{\lambda_{2,n}}.
      \end{equation}
   Fix this notation for the next results.
   

  \begin{lem}\label{lemacomb}
  Let $l,m,n\in\N$ be such that $1\leq l\leq n$ and $m\leq n-l+1$. Let $\eta\in\Omega$. Suppose that $E=\{(c_{1},l),\ldots,(c_{m},l)\}$ (resp. $\{(l,c_{1}),\ldots,(l,c_{m})\}$) is contained in $T_{\eta}$, for some $0<c_{1}<\cdots<c_{m}$. Moreover, suppose that for each $u\in T_{\eta}\setminus E$ such that $\pi_{2}(u)\geq l$ (resp. $\pi_{1}(u)\geq l$), we have that $a_{\bar{u}}=0$. Then for all $v\in E$ we obtain that $a_{\bar{v}}=0$.
  
  
  \end{lem}
  
  \begin{proof} Consider the set of vectors $D=\{(0,l),(1,l),\ldots,(n-l,l)\}\subset\Lambda_{2,n}$ (resp. $\{(l,0),(l,1),\ldots,(l,n-l)\}$). Let $u\in T_{\eta}\setminus E$. If $\pi_{2}(u)<l$ (resp. $\pi_{1}(u)<l$), then $\binom{u}{\alpha}=0$ for all $\alpha\in D$. If $\pi_{2}(u)\geq l$ (resp. $\pi_{1}(u)\geq l$), by hypothesis $a_{\bar{u}}=0$. Consider $\pi_{\alpha}:\C^{\lambda_{2,n}}\rightarrow\C$  the projection on the $\alpha$-th coordinate.
  Therefore, $\pi_{\alpha}(a_{\bar{u}}\bar{u})=0$ for all $u\in T_{\eta}\setminus E$ and $\alpha\in D$. This implies $$\sum_{v\in E}\pi_{\alpha}(a_{\bar{v}}\bar{v})=\sum_{\bar{v}\in\overline{T_{\eta}}}\pi_{\alpha}(a_{\bar{v}}\bar{v})=0,$$
 for all $\alpha\in D$. 
 
Since $\alpha=(j,l)$ (resp. $(l,j)$) with $0\leq j\leq n-l$ and $v=(c_{i},l)$ (resp. $(l,c_{i}$)), with $1\leq i\leq m$, we obtain that $\pi_{\alpha}(\bar{v})=\binom{c_{i}}{j}$. Thus $$\sum_{i=1}^{m}a_{\bar{v}}\binom{c_{i}}{j}=\sum_{v\in E}\pi_{\alpha}(a_{\bar{v}}\bar{v})=0,$$
 for all $0\leq j\leq n-l$. By lemma \ref{lemacomb1}, we obtain that $a_{\bar{v}}=0$ for all $v\in E$.

  \end{proof}
  
 \begin{lem}\label{des.ult}
 Let $\eta=(z,d_0,d_{1},\ldots,d_{r})\in\Omega$ and $1\leq l< j\leq n$. 
 \begin{itemize}
     \item
 If  $v_{l,\eta}=(p_{l},0)$  and $v_{j,\eta}=(p_{j},0)$, then $$\pi_{1}(v_{j,\eta}+(n-j)(1,1))\leq\pi_{1}(v_{l,\eta}+(n-l)(1,1)).$$
 The equality holds if and only if there exists $1\leq t\leq r$ such that $\sum_{i=0}^{t-1}d_{i}<l<j\leq\sum_{i=0}^{t}d_{i}$.
 
 \item If $v_{l,\eta}=(0,p_{l})$  and $v_{j,\eta}=(0,p_{j})$, then $$\pi_{2}(v_{j,\eta}+(n-j)(1,1))\leq\pi_{2}(v_{l,\eta}+(n-l)(1,1)).$$
 The equality holds if and only if there exists $1\leq t\leq r$ such that $\sum_{i=0}^{t-1}d_{i}<l<j\leq\sum_{i=0}^{t}d_{i}$.
 \end{itemize}
 \end{lem}
  \begin{proof}
  Suppose that $z=1$. By definition \ref{Teta} and the fact $l<j$, $p_{l}=\sum_{\substack{i\,\,\, \text{odd}\\ i<t}}d_{i}+c_{t}$ and $p_{j}=\sum_{\substack{i\,\,\, \text{odd}\\ i<t'}}d_{i}+c_{t'}$, for some odd numbers $t\leq t'\leq r$, where $c_{t}\leq d_{t}$ and $c_{t'}\leq d_{t'}$. Moreover, by definition $l=\sum_{i=0}^{t-1}d_{i}+c_{t}$ and $j=\sum_{i=0}^{t'-1}d_{i}+c_{t'}$. Then 
  \begin{align*}
\pi_{1}(v_{j,\eta}+(n-j)(1,1))&=p_{j}+(n-j)\\
&=n-\sum_{\substack{i\,\,\, \text{even}\\ i<t'}}d_{i}\\
&\leq n-\sum_{\substack{i\,\,\, \text{even}\\ i<t}}d_{i}\\ 
&=p_{l}+(n-l)\\
&=\pi_{1}(v_{l,\eta}+(n-l)(1,1)).
  \end{align*}
  Notice that the equality holds if and only if $t'=t$. For the other three cases ($z=1$, $v_{l,\eta}=(0,p_{l})$, $v_{j,\eta}=(0,p_{j})$; $z=0$, $v_{l,\eta}=(p_{l},0)$, $v_{j,\eta}=(p_{j},0)$; $z=0$, $v_{l,\eta}=(0,p_{l})$, $v_{j,\eta}=(0,p_{j})$) the proof is analogous. 
  \end{proof}
  
  Now we are ready to prove the first important result of the section.
  
  \begin{pro}\label{indepen}
  Let $\eta\in\Omega$. Then $\overline{T_{\eta}}$ is linearly independent.
  \end{pro}
  
  \begin{proof} Let $\eta=(z,d_{0},d_{1},\ldots,d_{r})$ and suppose that $z=1$. Define the numbers $$d_{+,r}=\sum_{\substack{i\leq r \\ i\,\,\, \text{odd}}}d_{i}, \,\,\,\,\,\,\,\,\,\,\,\, d_{-,r}=\sum_{\substack{i\leq r \\ i\,\,\, \text{even}}}d_{i}.$$
 Notice that by definition \ref{chain}, we have that $n=d_{+,r}+d_{-,r}$. We claim that for all $v\in T_{\eta}$ such that $\pi_{2}(v)>n-d_{+,r}$ or $\pi_{1}(v)>n-d_{-,r}$, we obtain that $a_{\bar{v}}$=0 in (\ref{equation5}). Assume this claim for the moment. For each $0\leq s\leq d_{+,r}$, define the set $E_{s}=\{v\in T_{\eta}\mid \pi_{1}(v)=s\,\,\,\mbox{and}\,\,\,\pi_{2}(v)\leq d_{-,r}\}$. Notice that $$|E_{s}|\leq d_{-,r}+1=n-d_{+,r}+1\leq n-s+1.$$ 
 Using the claim and taking $s=d_{+,r}$ we obtain the conditions of lemma \ref{lemacomb}. Thus $a_{\bar{v}}=0$ for all $v\in E_{d_{+,r}}$. Now we can repeat the same argument for $s=d_{+,r}-1$.
 Applying this process in a decreasing way for each $s\in \{0,\ldots,d_{r,+}\}$ we obtain that $a_{\bar{v}}=0$ for all $v\in\cup_{s=0}^{d_{+,r}}E_{s}$. Then for $v\in T_{\eta}$, we have three possibilities: $v\in\cup_{s=0}^{d_{+,r}}E_{s}$, $\pi_{1}(v)>d_{+,r}$, or $\pi_{2}(v)>d_{-,r}$. In any case, we obtain that $a_{\bar{v}}=0$ by the previous argument or the claim. This implies that  $\overline{T_{\eta}}$ is linearly independent.
 
 Now we proceed to prove the claim. For each $1\leq l\leq r$, define $$d_{+,l}=\sum_{\substack{i\leq l \\ i\,\,\, \text{odd}}}d_{i}, \,\,\,\,\,\,\,\,\,\,\,\, d_{-,l}=\sum_{\substack{i\leq l \\ i\,\,\, \text{even}}}d_{i}.$$
 We prove the claim by induction on $l$.
 By definition, we have that $d_{+,1}=d_{1}$ and $d_{-,1}=0$. Therefore we only have to prove that if $\pi_{2}(v)>n-d_{1}$, then $a_{\bar{v}}=0$. We claim that for all $v\in T_{\eta}$ such that $\pi_{2}(v)>n-d_{1}$, we have that $\pi_{1}(v)\geq\pi_{2}(v)$. We proceed to prove this claim by contrapositive. Let $v\in T_{\eta}$ be such that $\pi_{2}(v)>\pi_{1}(v)$. This implies that $v=(0,\pi_{2}(v)-\pi_{1}(v))+\pi_{1}(v)(1,1)=v_{j,\eta}+\pi_{1}(v)(1,1)$ for some $j\leq n$, where $\pi_{1}(v)\leq n-j$ by definition \ref{Teta}. By $5)$ of lemma \ref{rem2}, there exist $i<j$ such that $v_{i,\eta}=(0,1)$. Moreover, by definition \ref{Teta}, $i=d_{1}+1$.  By lemma \ref{des.ult}, we obtain
 \begin{align*}\pi_{2}(v)&=\pi_{2}(v_{j,\eta}+\pi_{1}(v)(1,1))\\
 &\leq \pi_{2}(v_{j,\eta}+(n-j)(1,1))\\
 & \leq \pi_{2}(v_{d_{1}+1,\eta}+(n-d_{1}-1)(1,1))\\
 &= n-d_{1},
 \end{align*}
as we claim. For each $s\in\{n-d_{1}+1,\ldots,n\}$, we define the set $E(s)=\{v\in T_{\eta}|\pi_{2}(v)=s\}$. By $4)$ of lemma \ref{rem2} and the previous claim, we have that for each $s\in\{n-d_{1}+1,\ldots,n\}$ we have $|E(s)|\leq n-s+1$. Now we are in the conditions of lemma \ref{lemacomb}. Applying the lemma for each $s$ in a descendant way, we obtain the result.
 
 Now suppose that the claim is true for $l$, i.e., for all $v\in T_{\eta}$ such that $\pi_{2}(v)>n-d_{+,l}$ or $\pi_{1}(v)>n-d_{-,l}$ for some $l\geq1$, we have that $a_{\bar{v}}=0$ and we prove for $l+1$. We have two cases: $l$ odd or $l$ even. We prove the case $l$ odd, the other case is analogous. Since $l$ is odd, we obtain that $d_{+,l}=d_{+,l+1}$ and $d_{-,l}+d_{l+1}=d_{-,l+1}$. Then, by the induction hypothesis, we only need to check that for all $v\in T_{\eta}$ such that $n-d_{-,l+1}<\pi_{1}(v)\leq n-d_{-,l}$ and $\pi_{2}(v)\leq n-d_{+,l}$, we have $a_{\bar{v}}=0$. For this, we are going to apply lemma \ref{des.ult} in an iterative way. By definition, $v_{\sum_{i=0}^{l}d_{i},\eta}=(d_{+,l},0)$. We claim that for all $v\in T_{\eta}$ such that $\pi_{1}(v)\geq n-d_{-,l+1}+1$, we have that $\pi_{2}(v)>\pi_{1}(v)-d_{+,l}-1$. We proceed to prove this claim by contrapositive. Let $v\in T_{\eta}$ be such that $\pi_{2}(v)\leq\pi_{1}(v)-d_{+,l}-1$. This implies that $v=(\pi_{1}(v)-\pi_{2}(v),0)+\pi_{2}(v)(1,1)=v_{j,\eta}+\pi_{2}(v)(1,1),$ where $\pi_{2}(v)\leq n-j$ by definition \ref{Teta}. Since $\pi_{1}(v)-\pi_{2}(v)\geq d_{+,l}+1$, by $5)$ of lemma \ref{rem2}, there exist $i<j$ such that $v_{i,\eta}=(d_{+,l}+1,0)$. Moreover, by definition \ref{Teta}, $i=\sum_{i=0}^{l+1}d_{i}+1$.  By lemma \ref{des.ult}, we obtain
 \begin{align*}
     \pi_{1}(v)& =\pi_{1}(v_{j,\eta}+\pi_{2}(v)(1,1))\\
     & \leq \pi_{1}(v_{j,\eta}+(n-j)(1,1))\\
     &\leq \pi_{1}(v_{\sum_{i=0}^{l+1}d_{i}+1,\eta}-(n-\sum_{i=0}^{l+1}d_{i}-1)(1,1))\\
     &= d_{+,l}+1+n-\sum_{i=0}^{l+1}d_{i}-1\\
     &<n-d_{-,l+1}+1
 \end{align*}
 as we claim. For each $s\in\{n-d_{-,l+1}+1,\ldots,n-d_{-,l}\}$, we define the set $E(s)=\{v\in T_{\eta}|\pi_{1}(v)=s\,\,\,\mbox{and}\,\,\,\pi_{2}(v)\leq n-d_{+,l}\}$. Notice that, by the previous claim, we have that for each $s\in\{n-d_{-,l+1}+1,\ldots,n-d_{-,l}\}$, $|E(s)|\leq (n-d_{+,l})-(s-d_{+,l}-1)=n-s+1$. By the induction hypothesis we are in the conditions of lemma \ref{lemacomb} for $s=n-d_{-,l}$. Applying the lemma for each $s$ in a descendant way, we obtain the result.
 
 In the case $z=0$ the claim becomes: for each $v\in T_{\eta}$ such that $\pi_{2}(v)>n-d_{-,r}$ or $\pi_{1}(v)>n-d_{+,r}$ then $a_{\bar{v}}=0$. The proof of this case is analogous.

  \end{proof}
  
  
   \end{subsection}
  
  \begin{subsection}{Moving $T_{j,\eta}$ along a diagonal preserves linear independence}

 Proposition \ref{indepen} shows that $\overline{T_{\eta}}$ is a basis of $\C^{\lambda_{2,n}}$ for all $\eta\in\Omega$. Our following goal is to show that we can move the set $T_{j,\eta}$ along a diagonal without losing the linear independence for all $j\in\{1,\ldots, n\}$. First we need the following combinatorial identities.
 
 \begin{lem}{\cite[Chapter 1]{Rio}}\label{combi} Given $n,m,p\in\N$, we have the following identities:
\begin{itemize}
    \item[1)] $\binom{n}{m}\binom{m}{p}=\binom{n}{p}\binom{n-p}{m-p}$.
    \item[2)]$\sum_{j}(-1)^{j}\binom{n-j}{m}\binom{p}{j}=\binom{n-p}{m-p}=\binom{n-p}{n-m}$.
    \item[3)] $\sum_{j}\binom{n}{m-j}\binom{p}{j}=\binom{n+p}{m}$.
    \item[4)] $\sum_{j}\binom{n-p}{m-j}\binom{p}{j}=\binom{n}{m}.$
\end{itemize}

\end{lem}

  \begin{lem}\label{depenrecta}
  For all $m\in\N$, we have that \small $\overline{(m,m)}\in\spn_{\C}\{\overline{(1,1)},\ldots,\overline{(n,n)}\}$. 
  \normalsize
  \end{lem}
  
  \begin{proof}
  Recalling notation \ref{notacion}, consider the vector $$v_{j}=\sum_{i=1}^{j}(-1)^{j-i}\binom{j}{i}\overline{(i,i)},$$ 
  for each $j\in\{1,\ldots,n\}$. Notice that for all $j\in\{1,\ldots,n\}$, we have $v_{j}\in\spn_{\C}\{\overline{(1,1)},\ldots,\overline{(n,n)}\}$. We claim that $\overline{(m,m)}=\sum_{j=1}^{n}\binom{m}{j}v_{j}$. We have to prove the identity:
  $$\binom{m}{p-q}\binom{m}{q}=\sum_{j=1}^{n}\sum_{i=1}^{j}(-1)^{j-i}\binom{m}{j}\binom{j}{i}\binom{i}{q}\binom{i}{p-q},$$
   for all $1\leq p \leq n$ and $0\leq q\leq p$. By $1)$ of lemma \ref{combi}, we obtain the identities:
  \begin{align}
\sum_{j=1}^{n}\sum_{i=1}^{j}(-1)^{j-i}\binom{m}{j}\binom{j}{i}\binom{i}{q}\binom{i}{p-q} & = \notag\\
\sum_{j=1}^{n}\sum_{i=1}^{j}(-1)^{j-i}\binom{m}{j}\binom{j}{q}\binom{j-q}{i-q}\binom{i}{p-q} & = \notag\\
\sum_{j=1}^{n}\sum_{i=1}^{j}(-1)^{j-i}\binom{m}{q}\binom{m-q}{j-q}\binom{j-q}{i-q}\binom{i}{p-q} &  = \notag\\
\binom{m}{q}\sum_{j=1}^{n}\sum_{i=1}^{j}(-1)^{j-i}\binom{m-q}{j-q}\binom{j-q}{i-q}\binom{i}{p-q}. &\notag
\end{align}
With this, the claim is reduced to prove that
  $$\binom{m}{p-q}=\sum_{j=1}^{n}\sum_{i=1}^{j}(-1)^{j-i}\binom{m-q}{j-q}\binom{j-q}{i-q}\binom{i}{p-q}.$$
Now we have the following identities, where the second identity comes from the rearrangement of the coefficients and the fourth identity by $2)$ of lemma \ref{combi}.
\begin{align}
\sum_{j=1}^{n}\sum_{i=1}^{j}(-1)^{j-i}\binom{m-q}{j-q}\binom{j-q}{i-q}\binom{i}{p-q} & = \notag \\
\sum_{j=1}^{n}(-1)^{j}\binom{m-q}{j-q}\Big( \sum_{i=1}^{j}(-1)^{i}\binom{j-q}{i-q}\binom{i}{p-q}\Big)& = \notag \\
\sum_{j=1}^{n}(-1)^{j}\binom{m-q}{j-q}\Big( \sum_{i}(-1)^{j-i}\binom{j-i}{p-q}\binom{j-q}{i}\Big)& = \notag \\
\sum_{j=1}^{n}(-1)^{j}\binom{m-q}{j-q}(-1)^{j}\Big( \sum_{i}(-1)^{i}\binom{j-i}{p-q}\binom{j-q}{i}\Big)& = \notag \\
\sum_{j=1}^{n}\binom{m-q}{j-q}\binom{q}{p-j}.
\end{align}
Finally, we have the following identities, where the first identity comes from replace $j$ by $j+q$ and the second by $3)$ of lemma \ref{combi},
$$\sum_{j=1}^{n}\binom{m-q}{j-q}\binom{q}{p-j}= \sum_{j=1}^{n}\binom{m-q}{j}\binom{q}{(p-q)-j}= \binom{m}{p-q},$$
 proving the claim. 

\end{proof}

\begin{lem}\label{lema11}For all $a,r\in\N$ and $l\leq n$, we have that
$$\sum_{i=0}^{l}(-1)^{i}\binom{l}{i}\sum_{j=0}^{n-l+1}(-1)^{n-l+1+j}\binom{n-l+1}{j}\overline{(a+r+j,r+i+j)}=\bar{0}.$$

\end{lem}

\begin{proof}The proof is by induction on $r$. First, consider $r=0$, we need to show that for all $(p-q,q)\in\Lambda_{2,n}$, 
$$\sum_{i=0}^{l}(-1)^{i}\binom{l}{i}\sum_{j=0}^{n-l+1}(-1)^{n-l+1+j}\binom{n-l+1}{j}\binom{a+j}{p-q}\binom{i+j}{q}=0.$$

First, notice 
\begin{align}\label{eq2.13.1}
    \sum_{i=0}^{l}(-1)^{i}\binom{l}{i}\sum_{j=0}^{n-l+1}(-1)^{n-l+1+j}\binom{n-l+1}{j}\binom{a+j}{p-q}\binom{i+j}{q} & = \notag\\
    \sum_{i=0}^{l}\sum_{j=0}^{n-l+1}(-1)^{n-l+1+j+i}\binom{n-l+1}{j}\binom{a+j}{p-q}\binom{i+j}{q}\binom{l}{i} & = \notag\\
    \sum_{j=1}^{n-l+1}(-1)^{n-l+1+j}\binom{n-l+1}{j}\binom{a+j}{p-q}\Big( \sum_{i=0}^{l}(-1)^{i}\binom{i+j}{q}\binom{l}{i}\Big). 
\end{align}
Now we have the following identity, where the first identity comes from the rearrangement of the sum and the second comes from $2)$ of lemma \ref{combi},
$$\sum_{i=0}^{l}(-1)^{i}\binom{i+j}{q}\binom{l}{i}=(-1)^{l}\sum_{i=0}^{l}(-1)^{i}\binom{l}{i}\binom{l+j-i}{q}=(-1)^{l}\binom{j}{q-l},$$
Replacing this identity in the sum (\ref{eq2.13.1}) and using $1)$ of lemma \ref{combi}, we obtain 
\begin{align}\label{eq2.13.2}
    \sum_{j=0}^{n-l+1}(-1)^{n+1+j}\binom{n-l+1}{j}\binom{a+j}{p-q}\binom{j}{q-l} & = \notag\\
   \sum_{j=0}^{n-l+1}(-1)^{n+1+j}\binom{n-l+1}{q-l}\binom{n-q+1}{j-q+l}\binom{a+j}{p-q} & = \notag\\
    (-1)^{n+1}\binom{n-l+1}{q-l}\sum_{j=0}^{n+1+j}(-1)^{j}\binom{n-q+1}{(n-l+1)-j}\binom{a+j}{p-q}. 
\end{align}
Replacing $i$ by $n-l+1-j$ on the sum \ref{eq2.13.2} and using $2)$ of lemma \ref{combi}, we obtain that
\begin{align}
    \sum_{j=0}^{n+1+j}(-1)^{j}\binom{n-q+1}{(n-l+1)-j}\binom{a+j}{p-q} & = \notag\\
    \sum_{j=0}^{n+1+j}(-1)^{n-l+1-j}\binom{n-q+1}{j}\binom{a+(n-l+1)-j}{p-q} & = \notag\\
    \binom{a+q-l}{p-n-1}. &  \notag
    \end{align}
Since $p\leq n$, the claim is true for $r=0$.

Now suppose that is true for $r-1$, i.e., $$\sum_{i=0}^{l}(-1)^{i}\binom{l}{i}\sum_{j=0}^{n-l+1+j}\binom{n-l+1}{j}\binom{a+(r-1)+j}{p-q}\binom{(r-1)+i+j}{q}=0,$$
and we have to show that $$\sum_{i=0}^{l}(-1)^{i}\binom{l}{i}\sum_{j=0}^{n-l+1+j}\binom{n-l+1}{j}\binom{a+r+j}{p-q}\binom{r+i+j}{q}=0,$$
for all $(p-q,q)\in\Lambda_{2,n}$.

Using basic properties of binomial coefficients, we have 
\small
\begin{align}
 &   \binom{a+r+j}{p-q}\binom{r+i+j}{q}= \notag\\
  & \binom{a+(r-1)+j}{p-q}\binom{(r-1)+i+j}{q}+\binom{a+(r-1)+j}{p-q}\binom{(r-1)+i+j}{q-1}+ \notag\\
  & \binom{a+(r-1)+j}{p-q-1}\binom{(r-1)+i+j}{q}+\binom{a+(r-1)+j}{p-q-1}\binom{(r-1)+i+j}{q-1}.\notag
\end{align}
\normalsize
Then
\small
\begin{align}
    & \sum_{i=0}^{l}(-1)^{i}\binom{l}{i}\sum_{j=0}^{n-l+1+j}\binom{n-l+1}{j}\binom{a+r+j}{p-q}\binom{r+i+j}{q}\notag\\
    =& \sum_{i=0}^{l}(-1)^{i}\binom{l}{i}\sum_{j=0}^{n-l+1+j}\binom{n-l+1}{j}\binom{a+(r-1)+j}{p-q}\binom{(r-1)+i+j}{q}\notag\\
    +&\sum_{i=0}^{l}(-1)^{i}\binom{l}{i}\sum_{j=0}^{n-l+1+j}\binom{n-l+1}{j}\binom{a+(r-1)+j}{p-q}\binom{(r-1)+i+j}{q-1}\notag\\
    +&\sum_{i=0}^{l}(-1)^{i}\binom{l}{i}\sum_{j=0}^{n-l+1+j}\binom{n-l+1}{j}\binom{a+(r-1)+j}{p-q-1}\binom{(r-1)+i+j}{q}\notag\\
    +&\sum_{i=0}^{l}(-1)^{i}\binom{l}{i}\sum_{j=0}^{n-l+1+j}\binom{n-l+1}{j}\binom{a+(r-1)+j}{p-q-1}\binom{(r-1)+i+j}{q-1}=0.\notag
\end{align}
\normalsize
Notice that each element of $\{(p-l,l),(p-l,l-1),(p-l-1,l),(p-l-1,l-1)\}$ belongs to $\Lambda_{2,n}$ or has a negative entry.
In any case, by induction hypothesis, each of the four sums are zero, obtaining the result.

\end{proof}

\begin{coro}\label{coro11}For all $a,r\in\N$ and $l\leq n$, we have that
$$\sum_{i=0}^{l}(-1)^{i}\binom{l}{i}\sum_{j=0}^{n-l+1}(-1)^{n-l+1+j}\binom{n-l+1}{j}\overline{(r+i+j,a+r+j)}=\bar{0}.$$

\end{coro}
\begin{proof}We need to prove that $$\sum_{i=0}^{l}(-1)^{i}\binom{l}{i}\sum_{j=0}^{n-l+1+j}\binom{n-l+1}{j}\binom{r+i+j}{p-q}\binom{a+r+j}{q}=0,$$
for all $(p-q,q)\in\Lambda_{2,n}$. Notice that by definition of $\Lambda_{2,n}$, if $(p-q,q)\in\Lambda_{2,n}$ then $(q,p-q)\in\Lambda_{2,n}$. With this and the previous lemma we obtain the result.
\end{proof}

Now we are ready to show the other important result of this section. As we mentioned before, the goal is to show that we can move the sets $T_{j,\eta}$ along a diagonal without losing the linear independence. We are going to prove this with some additional properties.



\begin{pro}\label{gene}Let $\eta\in\Omega$ and $l\in\{1,\ldots,n\}$. Let $(r_{1},\ldots,r_{l})\in\N^{l}$ and $T_{i,\eta}+r_{i}:=\{v+(r_{i},r_{i})\mid v\in T_{i,\eta}\}$. Then, we have 
\small
$$\spn_{\C}\{\bar{v}\in\C^{\lambda_{2,n}}\mid v\in T_{0,\eta}\bigcup(\cup_{i=1}^{l}T_{i,\eta}+r_{i})\}=\spn_{\C}\{\bar{v}\in\C^{\lambda_{2,n}}\mid v\in \cup_{i=0}^{l}T_{i,\eta}\}.$$
\normalsize

In particular, $\overline{v_{l,\eta}+(r,r)}\in\spn_{\C}\{\bar{v}\in\C^{\lambda_{2,n}}\mid v\in\cup_{i=0}^{l}T_{i,\eta}\}$, for all  $r\in\N$.
\end{pro}

\begin{proof} 
Let $\eta\in\Omega$. The proof is by induction on $l$. Consider $l=1$. There are two cases, $v_{1,\eta}=(1,0)$ or $v_{1,\eta}=(0,1)$. Suppose that $v_{1,\eta}=(1,0)$. Consider the sums
$$f_{0,r}=\sum_{j=0}^{n}(-1)^{n+j}\binom{n}{j}\overline{(1+r+j,r+j)},$$
$$f_{1,r}=\sum_{j=0}^{n}(-1)^{n+j}\binom{n}{j}\overline{(1+r+j,1+r+j)}.$$
Applying lemma \ref{lema11} for $a=1$ and $l=1$, we obtain that 
 \begin{equation}\label{pro2.14.1}
     f_{1,r}-f_{0,r}=\bar{0},
 \end{equation}
for all $r\in\N$. 

 By  lemma \ref{depenrecta}, we have that 
$$\{\overline{(1+r+j,1+r+j)}\}_{j=0}^{n}\subset\spn_{\C}\{\bar{v}\in\C^{\lambda_{2,n}}\mid v\in T_{0,\eta}\},$$
for all $r,j\in\N$. In particular $f_{1,r}\in\spn_{\C}\{\bar{v}\in\C^{\lambda_{2,n}}\mid v\in T_{0,\eta}\}$. Moreover, since $v_{1,\eta}=(1,0)$, for $r=0$, we have that $f_{0,0}-\overline{(1+n,n)}\in\spn_{\C}\{\bar{v}\in\C^{\lambda_{2,n}}\mid v\in T_{1,\eta}\}$. Then $$\overline{(1+n,n)}=f_{1,0}-f_{0,0}+\overline{(1+n,n)}\in\spn_{\C}\{\bar{v}\in\C^{\lambda_{2,n}}\mid v\in T_{0\eta}\cup T_{1,\eta}\}.$$

Notice that the coefficient of $\overline{(1,0)}$ is not zero. By elemental results of linear algebra, we have that 
\begin{align*}\spn_{\C}\{\bar{v}\in\C^{\lambda_{2,n}}\mid v\in T_{0\eta}\cup T_{1,\eta}\}&=\\
\Big(\spn_{\C}\{\bar{v}\in\C^{\lambda_{2,n}}\mid v\in T_{0\eta}\cup T_{1,\eta}\}\setminus\{\overline{(1,0)}\}\Big)\cup\{\overline{(1+n,n)}\}&=\\
\spn_{\C}\{\bar{v}\in\C^{\lambda_{2,n}}\mid v\in T_{0\eta}\cup \Big( T_{1,\eta}+1\Big)\}.
\end{align*}
Applying the same argument for $r=1$ in (\ref{pro2.14.1}), we obtain that 
\begin{align*}\spn_{\C}\{\bar{v}\in\C^{\lambda_{2,n}}\mid v\in T_{0\eta}\cup T_{1,\eta}+1\}&=\\
\Big(\spn_{\C}\{\bar{v}\in\C^{\lambda_{2,n}}\mid v\in T_{0\eta}\Big)\cup T_{1,\eta}+1\}\setminus\{\overline{(2,1)}\}\cup\{\overline{(2+n,1+n)}\}&=\\
\spn_{\C}\{\bar{v}\in\C^{\lambda_{2,n}}\mid v\in T_{0\eta}\cup \Big( T_{1,\eta}+2\Big)\}.
\end{align*}
Repeating the argument $r_{1}$ times for each $r$ and putting together all the identities, we obtain that 
$$\spn_{\C}\{\bar{v}\in\C^{\lambda_{2,n}}\mid v\in T_{0\eta}\cup T_{1,\eta}\}=\spn_{\C}\{\bar{v}\in\C^{\lambda_{2,n}}\mid v\in T_{0\eta}\cup\Big( T_{1,\eta}+r_{1}\Big)\}.$$
This finish the proof for $l=1$ and $v_{1,\eta}=(1,0)$. For $v_{1,\eta}=(0,1)$ the proof is analogous using the corollary \ref{coro11}.

Now suppose that the statement is true for $l-1$ and let $(r_{1},\ldots,r_{l})\in\N^{l}$. We claim that \begin{multline*}
 \spn_{\C}\{\bar{v}\in\C^{\lambda_{2,n}}\mid v\in \cup_{i=0}^{l-1}T_{i,\eta}\cup\Big( T_{l,\eta}+r_{l}\Big)\}=\\
\spn_{\C}\{\bar{v}\in\C^{\lambda_{2,n}}\mid v\in \cup_{i=0}^{l}T_{i,\eta}\}.   
\end{multline*}
Assume this claim for the moment. By induction hypothesis, we have that \begin{multline*}
\spn_{\C}\{\bar{v}\in\C^{\lambda_{2,n}}\mid v\in T_{0,\eta}\bigcup(\cup_{i=1}^{l-1}T_{i,\eta}+r_{i})\}=\\
\spn_{\C}\{\bar{v}\in\C^{\lambda_{2,n}}\mid v\in \cup_{i=0}^{l-1}T_{i,\eta}\}.    
\end{multline*}
This implies that \begin{multline*}
    \spn_{\C}\{\bar{v}\in\C^{\lambda_{2,n}}\mid v\in T_{0,\eta}\bigcup(\cup_{i=1}^{l}T_{i,\eta}+r_{i})\}=\\
    \spn_{\C}\{\bar{v}\in\C^{\lambda_{2,n}}\mid v\in \cup_{i=0}^{l}T_{i,\eta}\}.
\end{multline*}

Now we proceed to prove the claim. There are two cases, $v_{l,\eta}=(a,0)$ or $v_{l,\eta}=(0,a)$, where $0<a\leq l$ by $3)$ of lemma \ref{rem2}. Suppose that $v_{l,\eta}=(a,0)$. 
For each $i\in\{0,\ldots,l\}$ and $r\in\N$, consider the sum
$$f_{i,r}=\sum_{j=0}^{n-l+1}(-1)^{n-l+1+j}\binom{n-l+1}{j}\overline{(a+r+j,i+r+j)}.$$

 Applying lemma \ref{lema11} for $l$ and $a$, we have that 
 \begin{equation}\label{prop2.14.2}
 \sum_{i=0}^{l}(-1)^{i}\binom{l}{i}f_{i,r}=\bar{0}.
     \end{equation}

By $6)$ of lemma \ref{rem2}  we have that $$\{(a-1,0),\ldots,(1,0),(0,1),\ldots,(0,l-a)\}=\{v_{i,\eta}\}_{i=1}^{l-1}.$$
Notice that if $ i=a$, then $\overline{(a+r+j)(1,1)}\in\spn_{\C}\{\bar{v}\subset\C^{\lambda_{2,n}}\mid v\in T_{0,\eta}\}$ by lemma \ref{depenrecta}. If $1\leq i< a$, then
$$(a-i,0)+(i+r+j,i+r+j)=(a+r+j,i+r+j),
$$
and if $a< i\leq l$, then
$$(0,i-a)+(a-i,a-i)+(i+r+j,i+r+j)=(a+r+j,i+r+j).
$$
By the induction hypothesis, we obtain that 
$$\{\overline{(a+r+j,i+r+j)}\}_{j=0}^{n-l+1}\subset\spn_{\C}\{\bar{v}\in\C^{\lambda_{2,n}}\mid v\in\cup_{i=0}^{l-1} T_{i,\eta}\},$$
for all $i\in\{1,\ldots,l\}$, $r\in\N$. 
    In particular  $f_{i,r}\in\spn_{\C}\{\bar{v}\in\C^{\lambda_{2,n}}\mid v\in\cup_{i=0}^{l-1}T_{i,\eta}\}$ for all $i\in\{1,\ldots,l\}$. Moreover, since $v_{l,\eta}=(a,0)$, for $r=0$, we have that $f_{0,0}-\overline{(a+n-l+1,n-l+1)}\in\spn_{\C}\{\bar{v}\in\C^{\lambda_{2,n}}\mid v\in T_{l,\eta}\}$. Then 
    \begin{multline*}\overline{(a+n-l+1,n-l+1)}=-\Big(\sum_{i=0}^{l}(-1)^{i}\binom{l}{i}f_{i,0}\Big)+\overline{(a+n-l+1,n-l+1)}\\
    \in\spn_{\C}\{\bar{v}\in\C^{\lambda_{2,n}}\mid v\in \cup_{i=0}^{l}T_{i,\eta}\}.
    \end{multline*}
Applying the same argument that the case $l=1$, we obtain 
\begin{align*}
    \spn_{\C}\{\bar{v}\in\C^{\lambda_{2,n}}\mid v\in \cup_{i=0}^{l}T_{i,\eta}\}&=\\
    \spn_{\C}\{\bar{v}\in\C^{\lambda_{2,n}}\mid v\in \cup_{i=0}^{l-1}T_{i,\eta}\cup T_{l,\eta}+1\}&=\\
    \spn_{\C}\{\bar{v}\in\C^{\lambda_{2,n}}\mid v\in \cup_{i=0}^{l-1}T_{i,\eta}\cup T_{l,\eta}+2\}&=\\
     \vdots & \\
    \spn_{\C}\{\bar{v}\in\C^{\lambda_{2,n}}\mid v\in \cup_{i=0}^{l-1}T_{i,\eta}\cup T_{l,\eta}+r_{l}\} &.
\end{align*}

Now suppose that $v_{l,\eta}=(0,a)$. In this case we have that 
$$\{(0,a-1),\ldots,(0,1),(1,1),(1,0),\ldots,(l-a,0)\}=\{v_{i,\eta}\}_{i=0}^{l-1}.$$
Obtaining  
$$\{\overline{(i+r+j,a+r+j)}\}_{j=0}^{n-l+1}\in\spn_{\C}\{\bar{v}\in\C^{\lambda_{2,n}}\mid v\in\cup_{i=0}^{l-1} T_{i,\eta}\}.$$
The proof is analogous using corollary \ref{coro11}.
\end{proof}

  \end{subsection}
\end{section}

\begin{section}{Proof of Theorem \ref{mainn}}
In this section we give the proof of the main theorem. We first associate to each $\eta\in\Omega$ a unique $J_{\eta}\in S_{A_{n}}$ with certain properties. Secondly, we construct a distinguished element $J_{\eta_{k}}$ for each $k\in\{1,\ldots,n\}$ and prove that there exists another element $J_{\eta}\in S_{A_{n}}$ with the same value with respect to an order function. Finally, we prove that $J_{\eta_{k}}$ is minimal in $S_{A_{n}}$ with respect to the previous function.


\begin{defi}\label{traslacion}
Let $\eta\in\Omega$, $\{v_{i,\eta}\}_{i=1}^{n}$ and $\{T_{i,\eta}\}_{i=1}^{n}\subset\N^{2}$ as in definition \ref{Teta}. Consider $r_{i,\eta}:=n\cdot\pi_{2}(v_{i,\eta})$ for all $i\in\{1,\ldots,n\}$. We define
$$T_{\eta}':=T_{0,\eta}\cup(\cup_{i=1}^{n}T_{i,\eta}+r_{i,\eta}),$$
where $T_{i,\eta}+r_{i,\eta}:=\{v+(r_{i,\eta},r_{i,\eta})\mid v\in T_{i,\eta}\}$.
\end{defi}

\begin{exam}Let $n=6$, $r=5$ and $\eta=(1,0,1,1,1,1,2)$. By definition \ref{Teta}, we have that $v_{1,\eta}=(1,0)$, $v_{2,\eta}=(0,1)$, $v_{3,\eta}=(2,0)$, $v_{4,\eta}=(0,2)$, $v_{5,\eta}=(3,0)$ and $v_{6,\eta}=(4,0)$. By definition, we obtain that $r_{1,\eta}=0, r_{2,\eta}=6, r_{3,\eta}=0$, $r_{4,\eta}=12$, $r_{5,\eta}=0$ and $r_{6,\eta}=0$. Thus 
$$T'_{\eta}=T_{0,\eta}\cup T_{1,\eta}\cup (T_{2\eta}+6)\cup T_{3,\eta}\cup (T_{4,\eta}+12)\cup T_{5,\eta}\cup T_{6,\eta},$$
where 
$$T_{2,\eta}+6=\{(6,7),(7,8),(8,9),(9,10),(10,11)\},$$
$$T_{4,\eta}+12=\{(12,14),(13,15),(14,16)\}.$$

\begin{figure}[h]
\caption{Example of $T'_{\eta}$, with $\eta=(1,0,1,1,1,1,2)$}
\centering
\includegraphics[scale=.23]{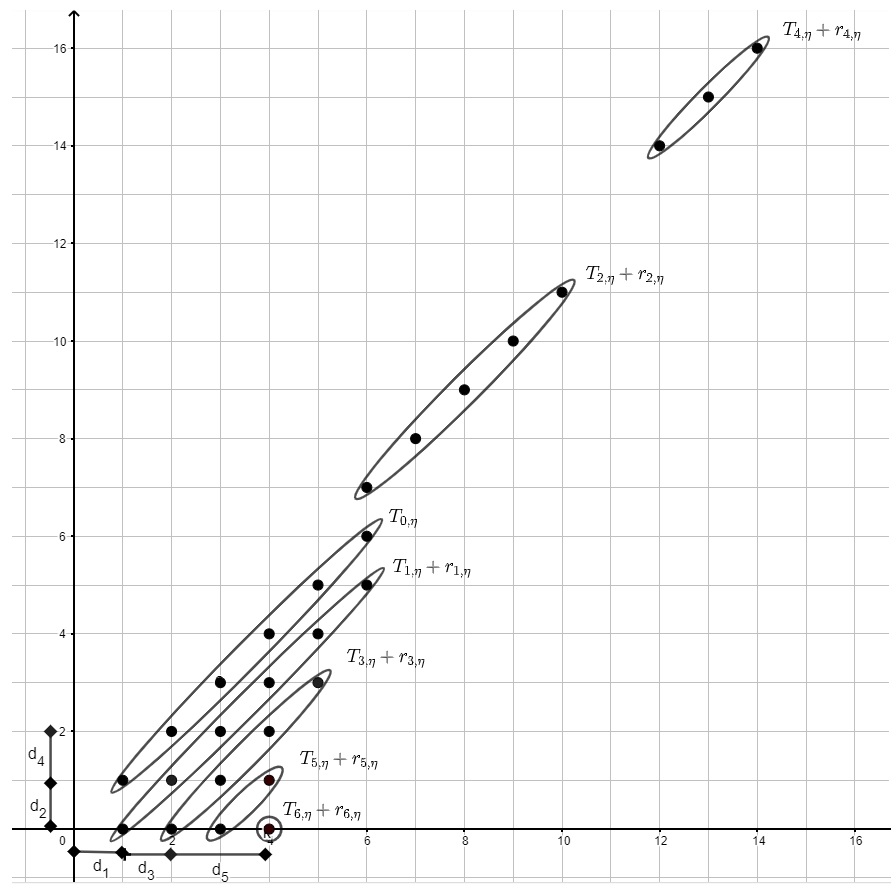}
\end{figure}

\end{exam}

\begin{rem}\label{unicidadbeta}
Recall notation \ref{notacion}. Let $\beta,\beta'\in\Lambda_{3,n}$ be such that $\beta\neq\beta'$. Then $A_{n}\beta\neq A_{n}\beta'$.
\end{rem}


\begin{pro}\label{corresomegasa}
For each $\eta\in\Omega$, there exists a unique $J_{\eta}\subset \Lambda_{3,n}$ such that $$A_{n}\cdot J_{\eta}:=\{A_{n}\cdot\beta\in\N^{2}\mid\beta\in J_{\eta}\}=T_{\eta}'.$$
Moreover, $J_{\eta}\in S_{A_{n}}$.
\end{pro}
\begin{proof}
We need to show that for each $v\in T_{\eta}'$, there exist an unique element $\beta\in\Lambda_{3,n}$ such that $A_{n}\beta=v$. The uniqueness comes from remark \ref{unicidadbeta}.

Now, let $v\in T_{\eta}'$. Then $v\in T_{\eta,0}$ or $v\in\cup_{i=1}^{n}T_{\eta,i}+r_{\eta,i}$. For the first case we have that $v_{0}=(t,t)$ with $t\leq n$. In this case we take $\beta=(0,t,0)$. For the second case we have 
$$v=v_{i,\eta}+(s,s)+r_{i,\eta}(1,1)=v_{\eta,i}+(s,s)+n\pi_{2}(v_{i,\eta})(1,1),$$
where $s\leq n-i$. By definition \ref{Teta}, $v_{i,\eta}=(q,0)$ or $v_{i,\eta}=(0,q)$, where $q\leq i$. Then 
$$v=(q+s,s)\,\,\,\,\,or\,\,\,\,\, v=(nq+s,(n+1)q+s).$$
 For these we take $\beta=(q,s,0)$ and $\beta=(0,s,q)$ respectively. Using the previous inequalities, we obtain that $\beta\in\Lambda_{3,n}$.
 
 Now we have to see that $J_{\eta}\in S_{A_{n}}$. Since $\lambda_{2,n}=|T_{\eta}|=|T'_{\eta}|=|J_{\eta}|$, we only have to see that $\det L_{J_{\eta}}^{c}\neq0$. Let $\{\be_{1},\be_{2},\ldots,\be_{\lambda_{2,n}}\}=J_{\eta}$ be such that $\be_{1}\prec\be_{2}\prec\cdots\prec\be_{\lambda_{2,n}}$, where $\prec$ denotes the lexicographic order. Notice that if $\be'<\be$ (see notation \ref{notacion}), then $\be'\prec\be$. By definition of $L_{J_{\eta}}^{c}$, we need to check that 
 $$\det\Big( \sum_{\ga\leq\be_{i}}(-1)^{|\be_{i}-\ga|}\binom{\be_{i}}{\ga}\overline{A_{n}\ga}\Big)_{1\leq i\leq\lambda_{2,n}}\neq0.$$
 For this, first we turn the previous matrix into $\left(
{\begin{array}{c}
 \overline{A\be_{i}} \\
\end{array}}
\right )_{1\leq i\leq \lambda_{2,n}}$ using elementary row operations. This implies the result since $A_{n}J_{\eta}=T'_{\eta}$ and $\{\overline{v}\in\Z^{\lambda_{2,n}}|v\in T_{\eta}'\}$ is linearly independent by proposition \ref{gene} and proposition \ref{indepen}.

Fix the $\lambda_{2,n}$-row. Consider $\be_{\lambda_{2,n}}\succ\ga_{1,\lambda_{2,n}}\succ\cdots\succ\ga_{r_{\lambda_{2,n}},\lambda_{2,n}}$, where $\{\ga_{i,\lambda_{2,n}}\}_{i=1}^{\lambda_{2,n}}=\{\ga\in\Lambda_{3,n}|\ga<\be_{\lambda_{2,n}}\}$. We can write this row as the sum 
 \begin{multline*}\overline{A_{n}\be_{\lambda_{2,n}}}+(-1)^{|\be_{\lambda_{2,n}}-\ga_{1,\lambda_{2,n}}|}\binom{\be_{\lambda_{2,n}}}{\ga_{1,\lambda_{2,n}}}\overline{A_{n}\ga_{1,\lambda_{2,n}}}+\\
 \cdots+(-1)^{|\be_{\lambda_{2,n}}-\ga_{r_{\lambda_{2,n}},\lambda_{2,n}}|}\binom{\be_{\lambda_{2,n}}}{\ga_{r_{\lambda_{2,n}},\lambda_{2,n}}}\overline{A_{n}\ga_{r_{\lambda_{2,n}},\lambda_{2,n}}},
      \end{multline*}

 Since $A_{n}\be_{\lambda_{2,n}}\in T_{\eta}'$ we have that $\be_{\lambda_{2,n}}$ have the shape $(q,s,0)$ or $(0,s,q)$ with $s+q\leq n$ and $A_{n}\be_{\lambda_{2,n}}$ equals one of $(q+s,s)$ or $(nq+s,(n+1)q+s)$.
Since $\ga_{1,\lambda_{2,n}}<\be_{\lambda_{2,n}}$, we obtain that $\ga_{1,\lambda_{2,n}}$ have the shape $(q',s',0)$ or $(0,s',q')$ with $s'<s$ or $q'<q$. Thus, $A_{n}\ga_{1,\lambda_{2,n}}$ have the shape $(q'+s',s')$ or $(nq'+s',(n+1)q'+s')$. In any case, we have that $A_{n}\ga_{1,\lambda_{2,n}}\in T'_{\eta}$. 

By the first part of the proposition, we have that $\ga_{1,\lambda_{2,n}}=\be_{i}$, for some $i<\lambda_{2,n}$. Then we subtract $(-1)^{|\be_{\lambda_{2,n}}-\ga_{1,\lambda_{2,n}}|}\binom{\be_{\lambda_{2,n}}}{\ga_{1,\lambda_{2,n}}}$-times the row $i$ to the row $\lambda_{2,n}$ in the matrix $L^{c}_{J_{\eta}}$. Notice that if $\ga<\be_{i}$, we have $\ga<\be_{\lambda_{2,n}}$. Thus we obtain that
$$\overline{A_{n}\be_{\lambda_{2,n}}}+c_{2}\overline{A_{n}\ga_{2,\lambda_{2,n}}}+\cdots+c_{r_{\lambda_{2,n}}}\overline{A_{n}\ga_{r_{\lambda_{2,n}},\lambda_{2,n}}}$$
 is the new $\lambda_{2,n}$-row, for some constants $\{c_{2},\ldots,c_{r_{\lambda_{2,n}}}\}\subset\Z$. Applying the same argument for each $\ga_{i,\lambda_{2,n}}$ in a increasing way, we turn the $\lambda_{2,n}$-th row into $\overline{A_{n}\be_{\lambda_{2,n}}}$.
 
 Applying this process to the other rows of $L_J^c$ in an ascending way we obtain the matrix
 $$\left(
{\begin{array}{c}
 \overline{A\be_{i}} \\
\end{array}}
\right )_{1\leq i\leq \lambda_{2,n}}.$$
 
 
\end{proof}




\begin{subsection}{A distinguished element of $S_{A_{n}}$}

Let $k\in\{1,\ldots,n\}$. Consider the function
\begin{align*}
  f_{k}: & \N^{2}\rightarrow \Z \\
  & v\mapsto \langle (k,1-k),v\rangle.
\end{align*}

\begin{defi}\label{minimum}Let $n\in\N\setminus\{0\}$, $1\leq k\leq n$ and $d_{k,0}=0$. If $f_{k}((1,0))\leq f_{k}((n,n+1))$, we take $z_{k}=1$. If $f_{k}((n,n+1))<f_{k}((1,0))$, we take $z_{k}=0$.

Now, we define $d_{k,l}$ for $l>0$ in an iterative way. Let
$$d_{k,l}=\min\{n-\sum_{j=0}^{l-1}d_{k,j},t_{l}-s_{l}\},$$
 where\small
   $$t_{l}= \left\{ \begin{array}{lcc}
                    \max\{m\in\N\mid m\cdot f_{k}((1,0))\leq f_{k}((\sum_{j\,\,even}^{l-1}d_{k,j}+1)(n,n+1))\} & \\
                    if\,\,\,\,z_{k}=1\,\,\, and\,\,\, l\,\,\,\mbox{odd}. &\\
             \\ \max\{m\in\N\mid m\cdot f_{k}((n,n+1))\leq f_{k}((\sum_{j\,\,odd}^{l-1}d_{k,j}+1)(1,0))\} &  \\
             if\,\,\,\,z_{k}=1\,\,\, and\,\,\, l\,\,\,\mbox{even}, &\\
             \\ \max\{m\in\N\mid m\cdot f_{k}((n,n+1))\leq f_{k}((\sum_{j\,\,even}^{l-1}d_{k,j}+1)(1,0))\} &  \\
             if\,\,\,\,z_{k}=0\,\,\, and\,\,\, l\,\,\,\mbox{odd}, &\\
            \\ \max\{m\in\N\mid m\cdot f_{k}((1,0))\leq f_{k}((\sum_{j\,\,odd}^{l-1}d_{k,j}+1)(n,n+1))\} &  \\ 
            if\,\,\,\,z_{k}=0\,\,\, and\,\,\, l\,\,\,\mbox{even}, &\\
             \end{array}
   \right.$$
   \normalsize
   and
   $$s_{l}= \left\{ \begin{array}{lcc}
                0 & if & l=1,\\
                \sum_{j\,\,odd}^{l-1}d_{k,j} &   if  &  l\,\,\,\mbox{odd}\,\,\,\mbox{and}\,\,\,l>1,\\
               \\  \sum_{j\,\,even}^{l-1}d_{k,j} &   if  &  l\,\,\,\mbox{even}.\\
            
             \end{array}
   \right.$$
   
   If $\sum_{j=1}^{l}d_{k,j}<n$, we define $d_{k,l+1}$. In other case, we finish the process and we define $\eta_{k}=(z_{k},d_{k,0},\ldots,d_{k,r})$. 

\end{defi}

\begin{exam}Let $n=6$ and $k=3$. We have that $d_{3,0}=0$. On the other hand, we have
$$f_{3}((1,0))=3<4=f_{3}((6,7)).$$
Then $z_{3}=1$. For $l=1$, we have that $$t_{1}=\max\{m\in\N\mid m\cdot 3=m\cdot f_{3}((1,0))\leq f_{3}((6,7))=4\}=1,$$
 and $s_{1}=0$. Then 
 $$d_{3,1}=\min\{6,1-0\}=1.$$
Now we computed $d_{3,2}$. By definition
 $$t_{2}=\max\{m\in\N\mid m\cdot 4 =m\cdot f_{3}((6,7))\leq f_{5}(2(1,0))=6\}=1,$$
 and $s_{2}=0$. This implies that 
 $$d_{3,2}=\min\{6-2=4,1-0\}=1.$$
In an analogous way we obtain that $d_{3,3}=1$ and $d_{3,4}=1$. Now we computed $d_{3,5}$. We have that
 $$t_{5}=\max\{m\in\N\mid m\cdot 3=m\cdot f_{3}((1,0))\leq f_{5}(3(6,7))=12\}=4,$$
and $s_{3}=d_{3,1}+d_{3,3}=2$. Then 
$$d_{3,5}=\min\{6-1-1-1-1=2,4-2=2\}=2.$$
Since $n-\sum_{j=0}^{3}d_{k,j}=6-2-1-3=0$ we finish the process. Thus $\eta_{5}=(1,0,1,1,1,1,2)$.

\end{exam}
\begin{lem}\label{propetak}Let $1\leq k\leq n$ and $\eta_{k}$ be as in definition \ref{minimum}. Then we have the following properties:
\begin{itemize}\item[1)]$d_{k,l}>0$ for all $l\in\{1,\ldots,r\}$. In particular, $\eta_{k}\in\Omega$. 
\item[2)]For each $i\in\{1,\ldots,n\}$, let $l\in\{1,\ldots,r\}$ be the unique element such that $\sum_{j=0}^{l-1}d_{k,j}<i\leq\sum_{j=1}^{l}d_{k,j}$. Then, we have the following inequalities:
\small
$$\begin{array}{lcc}
                f_{k}(v_{i,\eta_{k}})+r_{i,\eta_{k}}\leq f_{k}((\sum_{j\,\,even}^{l}d_{k,j}+1)(n,n+1)) &   if  \,\,\,z_{k}=1\,\,\, and\,\,\, l\,\,\mbox{odd},\\
             \\ f_{k}(v_{i,\eta_{k}})+r_{i,\eta_{k}}\leq f_{k}((\sum_{j\,\,odd}^{l}d_{k,j}+1)(1,0)) &   if  \,\,\, z_{k}=1\,\,\, and\,\,\, l\,\,\mbox{even},\\
             \\ f_{k}(v_{i,\eta_{k}})+r_{i,\eta_{k}}\leq f_{k}((\sum_{j\,\,even}^{l}d_{k,j}+1)(1,0)) &   if  \,\,\, z_{k}=0\,\,\, and\,\,\, l\,\,\mbox{odd},\\
            \\ f_{k}(v_{i,\eta_{k}})+r_{i,\eta_{k}}\leq f_{k}((\sum_{j\,\,odd}^{l}d_{k,j}+1)(n,n+1)) &   if  \,\,\, z_{k}=0\,\,\, and\,\,\, l\,\,\mbox{even}.\\
             \end{array}$$
             \normalsize
    \item[3)] Let $i',i\in\N\setminus\{0\}$ be such that $\sum_{j=0}^{l-1}d_{k,j}<i<i'\leq\sum_{j=0}^{l}d_{k,j},$
    for some $l\in\{1,\ldots,r\}$. Then 
    $$f_{k}(v_{i,\eta_{k}})+r_{i,\eta_{k}}\leq f_{k}(v_{i',\eta_{k}})+r_{i',\eta_{k}}.$$
    
    
    \item[4)]For all $1\leq i< i'\leq n$, we have that 
        $$f_{k}(v_{i,\eta_{k}}+r_{i,\eta_{k}}(1,1))\leq f_{k}(v_{i',\eta_{k}}+r_{i',\eta_{k}}(1,1)).$$

    \item[5)] If $l>2$ and $f_{k}(v_{l,\eta_{k}}+r_{l,\eta_{k}}(1,1))=f_{k}(v_{l-1,\eta_{k}}+r_{l-1,\eta_{k}}(1,1))$, then $f_{k}(v_{l,\eta_{k}}+r_{l,\eta_{k}}(1,1))\geq f_{k}(v_{l-2,\eta_{k}}+r_{l-2,\eta_{k}}(1,1))+2$         
\end{itemize}
\end{lem}

\begin{proof}\begin{itemize}
\item[1)]
By construction $n-\sum_{j=0}^{l-1}d_{k,j}>0$. Then, by definition \ref{minimum}, we only have to check that $t_{l}-s_{l}>0$.
Notice that by definition, $t_{1}>0$ and $s_{1}=0$. This implies that is true for $l=1$. Now suppose that $l>1$.

We have four cases: $z_{k}=1$ and $l$ odd; $z_{k}=1$ and $l$ even; $z_{k}=0$ and $l$ odd; $z_{k}=1$ and $l$ even. Consider $z_{k}=1$ and $l$ odd. By definition of $t_{l-1}$
$$f_{k}((t_{l-1}+1)(n,n+1))>f_{k}((\sum_{j\,\,odd}^{l-2}d_{k,j}+1)(1,0)).$$
Since $l$ is odd, $\sum_{j\,\,odd}^{l-2}d_{k,j}=\sum_{j\,\,odd}^{l-1}d_{k,j}$. It follows that 
$$f_{k}((\sum_{j\,\,odd}^{l-2}d_{k,j}+1)(1,0))=f_{k}((\sum_{j\,\,odd}^{l-1}d_{k,j}+1)(1,0))=f_{k}((s_{l}+1)(1,0)).$$

On the other hand, notice that if $d_{k,l-1}=n-\sum_{j=0}^{l-2}d_{k.j}$, then $n=\sum_{j=0}^{l-1}d_{k,j}$ and so there is no $d_{k,l}$, which is a contradiction. This implies that $d_{k,l-1}=t_{l-1}-s_{l-1}$. Thus 
$$f_{k}((t_{l-1}+1)(n,n+1))=f_{k}((d_{k,l-1}+s_{l-1}+1)(n,n+1)).$$
Since $l-1$ is even and $s_{l-1}=\sum_{j\,\, even}^{l-2}d_{k,j}$, we have that $d_{k,l-1}+s_{l-1}=\sum_{j\,\, even}^{l-1}d_{k,j}$. Then 
$$f_{k}((\sum_{j\,\, even}^{l-1}d_{k,j}+1)(n,n+1))>f_{k}((s_{l}+1)(1,0)).$$
By definition of $t_{l}$, we obtain that $t_{l}\geq s_{l}+1$ and so $t_{l}-s_{l}>0$. The other three cases are analogous.
\item[2)]Let $i\in\{1,\ldots,n\}$ and $l\in\{1,\ldots,r\}$. We have four cases: $z_{k}=1$ and $l$ odd; $z_{k}=1$ and $l$ even; $z_{k}=0$ and $l$ odd; $z_{k}=1$ and $l$ even. Suppose that $z_{k}=1$ and $l$ odd. In this case, by definition, $v_{i,\eta_{k}}=(\sum_{j\,\,odd}^{l-1}d_{k,j}+c)(1,0)$ with $c\leq d_{k,l}$, $r_{i,\eta_{k}}=0$ and $s_{l}=\sum_{j\,\,odd}^{l-1}d_{k,j}$. Then
$$\sum_{j\,\,odd}^{l-1}d_{k,j}+c\leq \sum_{j\,\,odd}^{l-1}d_{k,j}+d_{k,l}=s_{l}+d_{k,l}\leq t_{l}.$$
By definition of $t_{l}$, we have that

 $$   f_{k}(v_{i,\eta_{k}})+r_{i,\eta_{k}}=
    f_{k}((\sum_{j\,\,odd}^{l-1}d_{k,j}+c)(1,0))\leq
    f_{k}((\sum_{j\,\,even}^{l-1}d_{k,j}+1)(n,n+1)).$$
Since $l$ is odd, we have that $\sum_{j\,\,even}^{l-1}d_{k,j}=\sum_{j\,\,even}^{l}d_{k,j}$. This implies the inequality that we need.

Now suppose that $z_{k}=1$ and $l$ even. In this case, by definition, $v_{i,\eta_{k}}=(\sum_{j\,\,even}^{l-1}d_{k,j}+c)(0,1)$ with $c\leq d_{k,l}$, $r_{i,\eta}=n(\sum_{j\,\,even}^{l-1}d_{k,j}+c)$ and $s_{l}=\sum_{j\,\,even}^{l-1}d_{k,j}$. Using the above and the linearity of $f_{k}$ we obtain
\begin{align*}
    f_{k}(v_{i,\eta_{k}})+r_{i,\eta_{k}}&=
     f_{k}(v_{i,\eta_{k}})+f_{k}(r_{i,\eta_{k}}(1,1))\\
     &=f_{k}(v_{i,\eta_{k}}+r_{i,\eta_{k}}(1,1))\\
     &=f_{k}((\sum_{j\,\,even}^{l-1}d_{k,j}+c)(n,n+1)).
\end{align*}
Since
$$\sum_{j\,\,even}^{l-1}d_{k,j}+c\leq \sum_{j\,\,even}^{l-1}d_{k,j}+d_{k,l}\leq t_{l},$$
by definition of $t_{l}$, we obtain the inequality 
$$f_{k}(v_{i,\eta_{k}})+r_{i,\eta_{k}}=f_{k}((\sum_{j\,\,even}^{l-1}d_{k,j}+c)(n,n+1))\leq f_{k}((\sum_{j\,\,odd}^{l-1}d_{k,j}+1)(1,0)).$$

Since $l$ is even, we have that $\sum_{j\,\,odd}^{l-1}d_{k,j}=\sum_{j\,\,odd}^{l}d_{k,j}$, obtaining the result.

The other two cases are analogous.
\item[3)] The hypothesis implies that $i=\sum_{j=0}^{l-1}d_{k,j}+c_{i}$ and $i'=\sum_{j=0}^{l-1}d_{k,j}+c_{i'}$, where $0<c_{i}<c_{i'}\leq d_{k,l}$. We have four cases: $z_{k}=1$ and $l$ odd; $z_{k}=1$ and $l$ even; $z_{k}=0$ and $l$ odd; $z_{k}=1$ and $l$ even. Consider $z_{k}=1$ and $l$ even. By definition \ref{Teta}, $v_{i,\eta_{k}}=(0,\sum_{j\,\,even}^{l-1}d_{k,j}+c_{i})$ and $v_{i',\eta_{k}}=(0,\sum_{j\,\,even}^{l-1}d_{k,j}+c_{i'}).$ Then
\begin{align*}
    f_{k}(v_{i,\eta_{k}})+r_{i,\eta_{k}}&=
    (1-k)(\sum_{j\,\,even}^{l-1}d_{k,j}+c_{i})+n(\sum_{j\,\,even}^{l-1}d_{k,j}+c_{i})\\
    &=(n-k+1)(\sum_{j\,\,even}^{l-1}d_{k,j})+(n-k+1)c_{i}\\
     &\leq(n-k+1)(\sum_{j\,\,even}^{l-1}d_{k,j})+(n-k+1)c_{i'}\\
     &= (1-k)(\sum_{j\,\,even}^{l-1}d_{k,j}+c_{i'})+n(\sum_{j\,\,even}^{l-1}d_{k,j}+c_{i'})\\
      &=f_{k}(v_{i',\eta_{k}})+r_{i',\eta_{k}}.
\end{align*}
The other three cases are analogous.

\item[4)] Let $1\leq i< i'\leq n$. Let $1\leq l\leq l'\leq r$ be such that $i=\sum_{j=0}^{l-1}d_{k,j}+c_{i}$ and $i'=\sum_{j=0}^{l'-1}d_{k,j}+c_{i'}$. By hypothesis, we have that $l\leq l'$. If $l=l'$, the result follows from $3)$. Suppose that $l<l'$, this implies that $l'=l+c$ with $c>0$.  

We have four cases ($z_{k}=1$ and $l$ odd; $z_{k}=1$ and $l$ even; $z_{k}=0$ and $l$ odd; $z_{k}=1$ and $l$ even). Consider $z_{k}=0$ and $l$ even. By definition \ref{Teta}, we have that $v_{i,\eta_{k}}=(\sum_{j\,\,even}^{l-1}d_{k,j}+c_{i},0)$. By $2)$, we have that
$$f_{k}(v_{i,\eta_{k}})+r_{i,\eta_{k}}\leq f_{k}((\sum_{j\,\,odd}^{l}d_{k,j}+1)(n,n+1)).$$
On the other hand, by definition \ref{Teta}, we obtain that $v_{\sum_{j=0}^{l}d_{k,j}+1,\eta_{k}}=(0,\sum_{j\,\,odd}^{l}d_{k,j}+1)$.
Then

\begin{align*}
    f_{k}(v_{\sum_{j=0}^{l}d_{k,j}+1,\eta_{k}})+r_{\sum_{j=0}^{l}d_{k,j}+1,\eta_{k}}=\\
    f_{k}((0,\sum_{j\,\,odd}^{l}d_{k,j}+1)+n\cdot(\sum_{j\,\,odd}^{l}d_{k,j}+1)(1,1))=\\
    f_{k}((\sum_{j\,\,odd}^{l}d_{k,j}+1)(n,n+1))\geq \\
    f_{k}(v_{i,\eta_{k}})+r_{i,\eta_{k}}.
\end{align*}

Now, if $c>1$, by $2)$ and knowing that $l+1$ is odd, we obtain that 
$$ f_{k}(v_{\sum_{j=0}^{l}d_{k,j}+1,\eta_{k}})+r_{\sum_{j=0}^{l}d_{k,j}+1,\eta_{k}}\leq f_{k}((\sum_{j\,\,even}^{l+1}d_{k,j}+1)(1,0)).$$
In addition, using definition \ref{Teta}, we have the vector $v_{\sum_{j=0}^{l+1}d_{k,j}+1,\eta_{k}}=(\sum_{j\,\,even}^{l+1}d_{k,j}+1,0).$
Then 
\begin{align*}
    f_{k}(v_{\sum_{j=0}^{l}d_{k,j}+1,\eta_{k}})+r_{\sum_{j=0}^{l}d_{k,j}+1,\eta_{k}}\leq\\
     f_{k}((\sum_{j\,\, even}^{l+1}d_{k,j}+1)(1,0))=\\
      f_{k}(v_{\sum_{j=0}^{l+1}d_{k,j}+1,\eta_{k}})=\\
      f_{k}(v_{\sum_{j=0}^{l+1}d_{k,j}+1,\eta_{k}})+r_{\sum_{j=0}^{l+1}d_{k,j}+1,\eta_{k}}.
\end{align*}
Repeating this argument $c$ times, we obtain 

\begin{align*}
    f_{k}(v_{i,\eta_{k}})+r_{i,\eta_{k}}& \leq
    f_{k}(v_{\sum_{j=0}^{l}d_{k,j}+1,\eta_{k}})+r_{\sum_{j=0}^{l}d_{k,j}+1,\eta_{k}}\\
    &\leq f_{k}(v_{\sum_{j=0}^{l+1}d_{k,j}+1,\eta_{k}})+r_{\sum_{j=0}^{l+1}d_{k,j}+1,\eta_{k}}\\
    \vdots \\
    &\leq f_{k}(v_{\sum_{j=0}^{l'-1}d_{k,j}+1\eta_{k}})+r_{\sum_{j=0}^{l'-1}d_{k,j}+1\eta_{k}}\\
    & \leq f_{k}(v_{i',\eta_{k}})+r_{i',\eta_{k}},
    \end{align*}
where the last inequality comes from $2)$. The other cases are analogous.  

\item[5)] Notice that if $k=n$, $f_{n}(t(n,n+1))=t$ for all $t\in\{1,\ldots, n\}$ and $f_{n}((1,0))=n$. Then by definition \ref{minimum}, $\eta_{n}=(0,0,n)$, i.e., $v_{j,\eta_{n}}=(0,j)$ for all $j\in\{1,\ldots,n\}$. In particular, $f_{n}(v_{i,\eta_{n}}+r_{i,\eta_{n}})<f_{n}(v_{j,\eta_{n}}+r_{j,\eta_{n}})$ if $1\leq i<j\leq n$,. Then we cannot have the conditions of lemma. Analogous, if $k=1$, $\eta_{1}=(1,0,n)$, and $f_{1}(v_{i,\eta_{1}}+r_{i,\eta_{1}})<f_{1}(v_{j,\eta_{1}}+r_{j,\eta_{1}})$, for all $1 \leq i<j\leq n$. This implies that if there exists $l\in \{1,\ldots, n\}$ such that $f_{k}(v_{l,\eta_{k}}+r_{l,\eta_{k}})=f_{k}(v_{l-1,\eta_{k}}+r_{l-1,\eta_{k}})$, we have that $k\in\{2,\ldots,n-1\}$.

Now, suppose that there exist $l\in \{1,\ldots, n\}$ such that $f_{k}(v_{l,\eta_{k}}+r_{l,\eta_{k}})=f_{k}(v_{l-1,\eta_{k}}+r_{l-1,\eta_{k}})$. If $v_{l,\eta_{k}}=(0,s)$ and $v_{l-1,\eta_{k}}=(0,s-1)$, then 
\small
$$f_{k}(v_{l,\eta_{k}}+r_{l,\eta_{k}})=s(n-k+1)>(s-1)(n-k+1)=f_{k}(v_{l-1,\eta_{k}}+r_{l-1,\eta_{k}}).$$
\normalsize
In an analogous way, obtain a contradiction if $v_{l,\eta_{k}}=(t,0)$ and $v_{l-1,\eta_{k}}=(t-1,0)$. This implies that $v_{l,\eta_{k}}=(t,0)$ and $v_{l-1,\eta_{k}}=(0,s)$ or $v_{l,\eta_{k}}=(0,s)$ and $v_{l-1,\eta_{k}}=(t,0)$. Consider the first case, the other case is analogous. By definition
\begin{equation}\label{eq2}
    f_{k}(v_{l,\eta_{k}}+r_{l,\eta_{k}})=f_{k}((t,0))=f_{k}((0,s)+(ns,ns)),
\end{equation}

By $5)$ of lemma \ref{rem2}, we deduce that $v_{l-2,\eta_{k}}=(0,s-1)$ or $v_{l-2,\eta_{k}}=(t-1,0)$. Suppose that $v_{l-2,\eta_{k}}=(0,s-1)$. Then we have 
\begin{align*}
    f_{k}(v_{l,\eta_{k}}+r_{l,\eta_{k}})= & f_{k}((0,s)+(ns,ns)) \\
    = & f_{k}(s(n,n+1))\\
    = & s(n-k+1)\\
    = & (s-1)(n-k+1)+n-k+1\\
    = & f_{k}((s-1)(n,n+1))+n-k+1\\
    \geq & f_{k}(v_{l-2,\eta_{k}}+r_{l-2,\eta_{k}})+2,
\end{align*}
where the first equality comes from equation (\ref{eq2}) and the last inequality comes from $k\leq n-1$.

Now suppose that $v_{l-2,\eta_{k}}=(t-1,0)$. In an analogous way, we obtain that
\begin{align*}
    f_{k}(v_{l,\eta_{k}}+r_{l,\eta_{k}})= & f_{k}((t,0))\\
    = & k(t-1)+k\\
    \geq & f_{k}(v_{l-2,\eta_{k}}+r_{l-2,\eta_{k}})+2,
\end{align*}
where the first equality comes from equation (\ref{eq2}) and the last inequality comes from $k\leq 2$. Obtaining the result.
\end{itemize}

\end{proof}

The previous lemma will be constantly use in the rest of the section.

\begin{pro}\label{finalelem}
Let $k\{1,\ldots, n\}$. Let $\eta_{k}\in\Omega$ be from definition \ref{minimum}. Then $v_{n,\eta_{k}}=(0,k)$ or $v_{n,\eta_{k}}=(n-k+1,0)$.
\end{pro}

\begin{proof}
Suppose is not true. By definition, we have that $\sum_{j=0}^{r}d_{k,j} =n$. Using lemma \ref{rem2} and since $v_{n,\eta_{k}}$ is not $(n-k+1,0)$ or $(0,k)$, we have that there exist $m<n$ such that $v_{m,\eta_{k}}=(n-k+1,0)$ or $v_{m,\eta_{k}}=(0,k)$. Let $l\leq r$ be such that $m=\sum_{j=0}^{l-1}d_{k,j}+c$ and $0< c \leq d_{k,l}$. 

We have four cases: $z_{k}=1$ and $l$ odd; $z_{k}=1$ and $l$ even; $z_{k}=0$ and $l$ odd; $z_{k}=1$ and $l$ even. Consider $z=1$ and $l$ odd. In this case, by definition \ref{Teta},  $v_{m,\eta_{k}}=(\sum_{j\,\,odd}^{l-1}d_{k,j}+c,0)=(n-k+1,0)$. Hence $\sum_{j\,\,odd}^{l-1}d_{k,j}+c=n-k+1$. Then
$$n-k+1=m-\sum_{j\,\,even}^{l-1}d_{k,j}<n-\sum_{j\,\,even}^{l-1}d_{k,j}.$$
This implies that $\sum_{j\,\,even}^{l-1}d_{k,j}+1<k$. Thus
\begin{align*}    f_{k}((\sum_{j\,\,even}^{l-1}d_{k,j}+1)(n,n+1))& < f_{k}(k(n,n+1))\\
    & = k(n-k+1)\\
    & = f_{k}((n-k+1,0))\\
    & = f_{k}(v_{m,\eta_{k}})+r_{m,\eta_{k}}.
\end{align*}
This is a contradiction to lemma \ref{propetak} $2)$. 

Now, suppose that $z_{k}=1$ and $l$ even. For this case, by definition \ref{Teta}, we have that $v_{m,\eta_{k}}=(0,k)$ and $k=\sum_{j\,\,even}^{l-1}d_{k,j}+c$. Then
$$k=\sum_{j\,\,even}^{l-1}d_{k,j}+c=m-\sum_{j\,\,odd}^{l-1}d_{k,j}<n-\sum_{j\,\,odd}^{l-1}d_{k,j}.$$
This implies that $\sum_{j\,\,odd}^{l-1}d_{k,j}+1<n-k+1$. Thus
\begin{align*}
    f_{k}((\sum_{j\,\,odd}^{l-1}d_{k,j}+1)(1,0))& <f_{k}((n-k+1)(1,0))\\
    & =k(n-k+1)\\
    & =f_{k}(k(n,n+1))\\
    & =f_{k}((0,k)+k\cdot n(1,1))\\
    & =f_{k}(v_{m,\eta_{k}})+r_{m,\eta_{k}}.
    \end{align*}
    This is a contradiction to lemma \ref{propetak} $2)$. The other two cases are analogous. 
\end{proof}

Recall that for $J\in S_{A_{n}}$, we denote $m_{J}=\sum_{\be\in J}A_{n}\be$.

\begin{coro}\label{final1}
Let $n\in\N\setminus\{0\}$ and $1\leq k\leq n$. Then, there exists $\eta\in\Omega$ such that $\eta\neq\eta_{k}$ and $f_{k}(m_{J_{\eta}})=f_{k}(m_{J_{\eta_{k}}})$.
\end{coro}
\begin{proof}
By previous proposition, we have that $v_{n,\eta_{k}}=(n-k+1,0)$ or $v_{n,\eta_{k}}=(0,k)$. By lemma $6)$ of \ref{rem2}, we obtain that $\{v_{i,\eta_{k}}\}_{i=1}^{n-1}=\{(t,0)\}_{t=1}^{n-k}\cup\{(0,s)\}_{s=1}^{k-1}$. Moreover, we can deduce that $v_{n-1,\eta_{k}}$ is $(n-k,0)$ or $(0,k-1)$. 

Suppose that $v_{n-1,\eta_{k}}=(n-k,0)$. Since $f_{k}(k(n,n+1))=f_{k}((n-k+1,0))$ and by construction of $\eta_{k}$, we have that $v_{n,\eta_{k}}=(n-k+1,0)$. If $v_{n-1,\eta_{k}}=(0,k-1)$, we obtain that $v_{n,\eta_{k}}=(0,k)$. In any case, we obtain that $d_{k,r}\geq 2$. Then we define $\eta=(z',d_{0}',d_{1}',\ldots,d_{r}',d_{r+1}')$, where $z'=z_{k}$, $d_{i}'=d_{k,i}$ for all $i<r$, $d_{r}'=d_{k,r}-1$ and $d'_{r+1}=1$.

By construction $\sum_{j=0}^{n}d_{j}'=n$ and $d_{j}'>0$ for all $j\in\{1,\ldots,n\}$. This implies that $\eta\in\Omega$. On the other hand, we have that $v_{j,\eta_{k}}=v_{j,\eta}$ for all $j\leq n-1$ and $v_{n,\eta}=(n-k+1,0)$ if $v_{n,\eta_{k}}=(0,k)$ or $v_{n,\eta}=(0,k)$ if $v_{n,\eta_{k}}=(n-k+1,0)$. Since $f_{k}(k(n,n+1))=f_{k}((n-k+1,0))$, we obtain that $f_{k}(m_{J_{\eta}})=f_{k}(m_{J_{\eta_{k}}}).$
\end{proof}

\end{subsection}

\begin{subsection}{$J_{\eta_{k}}\in S_{A_{n}}$ is minimal with respect to $f_{k}$}

\begin{lem}\label{lema31} Let $\be',\be\in\N^{3}$ be such that $\be'\leq\be$ (Recall notation \ref{notacion}). Then 
$$f_{k}(A_{n}\be')\leq f_{k}(A_{n}\be).$$

\end{lem}

\begin{proof}This is a straightforward computation.

\end{proof}

\begin{lem}\label{lema32}Let $\be\in\N^{3}$ be such that $A_{n}\be\neq v+q(1,1)$ for all $v\in T_{\eta_{k}}'$ and $q\in\N$. Then $f_{k}(A_{n}\be)\geq f_{k}(v)$ for all $v\in T'_{\eta_{k}}$.

\end{lem}
\begin{proof}We claim that $f_{k}(v)\leq k(n-k+1)\leq f_{k}(A_{n}\be)$ for all $v\in T_{\eta_{k}}'$ and for all $\be\in\N^3$ with the conditions of the lemma.

We are going to prove the first inequality of the claim. By definition \ref{traslacion}, we have that $T_{\eta_{k}}'=T_{0,\eta_{k}}\bigcup\cup_{j=1}^{n}T_{j,\eta_{k}}+r_{j,\eta_{k}}$, where $T_{0,\eta_{k}}=\{(q,q)\}_{q=1}^{n}$, $r_{j,\eta_{k}}=n\cdot\pi_{2}(v_{j,\eta_{k}})$ and $T_{j,\eta_{k}}+r_{j,\eta_{k}}=\{v_{j,\eta_{k}}+(p+r_{j,\eta_{k}})(1,1)\}_{p=0}^{n-j}$ .  By proposition \ref{finalelem} we have that $v_{n,\eta_{k}}=(0,k)$ or $v_{n,\eta_{k}}=(n-k+1,0)$. Moreover, by $5)$ of lemma \ref{rem2}, we have that $\{v_{j,\eta_{k}}\}_{j=1}^{n-1}=\{(t,0)\}_{t=1}^{n-k}\cup\{(0,s)\}_{s=1}^{k-1}$. 

By definition, $T_{n,\eta_{k}}+r_{n,\eta_{k}}=\{v_{n,\eta_{k}}+r_{n,\eta_{k}}(1,1)\}$. Since we know the two possibilities for $v_{n,\eta_{k}}$, we obtain that $f_{k}(v_{n,\eta_{k}}+r_{n,\eta_{k}}(1,1))=k(n-k+1)$. On the other hand, if $v\in T_{0,\eta_{k}}$, we have that $v=(q,q)$ with $q\leq n$. Since $1\leq k\leq n$, obtaining that $f_{k}(v)=q\leq n\leq k(n-k+1)$. With this, we only have to check the desired inequality for $v\in \cup_{j=1}^{n-1}\{v_{j,\eta_{k}}+(p+r_{j,\eta_{k}})(1,1)\}_{p=0}^{n-j}$. This implies that $v=v_{j,\eta_{k}}+(p+r_{j,\eta_{k}})(1,1)$, for $1\leq j\leq n-1$ and $0\leq p\leq n-j$.

Suppose that $v_{j,\eta_{k}}=(t,0)$ for some $t\leq j$ and recall that, $t\leq n-k$. Then 
\begin{align*}
    f_{k}(v)= & f_{k}(v_{j,\eta_{k}}+(p+r_{j,\eta_{k}})(1,1))\\
    = & f_{k}((t,0))+p+r_{j,\eta_{k}}\\
    \leq & kt+n-j \\
    \leq & kt+n-t \\
    = & (k-1)t+n \\
    \leq & (k-1)(n-k)+n\\
    = & nk-k^{2}+k.
\end{align*}

Now suppose that $v_{j,\eta_{k}}=(0,s)$ for some $s\leq j$ and  recall that $s< k$. Then 
\begin{align*}
    f_{k}(v)= & f_{k}(v_{j,\eta_{k}}+(p+r_{j,\eta_{k}})(1,1))\\
    = & f_{k}((0,s)+p(1,1)+n\cdot s(1,1))\\
    = & f_{k}(s(n,n+1))+p \\
    \leq & s(n-k+1)+(n-j) \\
    \leq & s(n-k+1)+(n-s) \\
    \leq & s(n-k+1)+(n-k)+(k-s)\\
    \leq & s(n-k+1)+(k-s)(n-k)+(k-s)\\
    = & nk-k^{2}+k.
\end{align*}

This proves the first inequality of the claim. For the second inequality notice that 
\begin{align*}
    A_{n}\be= & \pi_{1}(\be)(1,0)+\pi_{2}(\be)(1,1)+\pi_{3}(\be)(n,n+1)\\
    = & (\pi_{1}(\be),0)+(\pi_{2}(\be),\pi_{2}(\be))+(n\pi_{3}(\be),n\pi_{3}(\be)+\pi_{3}(\be))\\
    = & (\pi_{1}(\be),\pi_{3}(\be))+(\pi_{2}(\be)+n\pi_{3}(\be))(1,1) \\
    = & (\pi_{1}(\be)-\pi_{3}(\be),0)+(\pi_{2}(\be)+(n+1)\pi_{3}(\be))(1,1).
\end{align*}

Similarly, we obtain the expression $$A_{n}\be=(0,\pi_{3}(\be)-\pi_{1}(\be))+(\pi_{1}(\be)+\pi_{2}(\be)+n\pi_{3}(\be))(1,1) ).$$
Working with the first expression of $A_{n}\be$ and applying $f_{k}$ to this vector, we obtain that

\begin{align*}
    f_{k}(A_{n}\be)= & f_{k}((\pi_{1}(\be)-\pi_{3}(\be),0)+(\pi_{2}(\be)+(n+1)\pi_{3}(\be))(1,1))\\
    = & f_{k}((\pi_{1}(\be)-\pi_{3}(\be),0))+\pi_{2}(\be)+(n+1)\pi_{3}(\be)\\
    = & k(\pi_{1}(\be)-\pi_{3}(\be))+\pi_{2}(\be)+(n+1)\pi_{3}(\be).
\end{align*}

By the hypothesis over $\be$ and recalling that $\{v_{\eta_{k},j}\}_{j=1}^{n-1}=\{(t,0)\}_{t=1}^{n-k}\cup\{(0,s)\}_{s=1}^{k-1}$, we obtain that $\pi_{1}(\be)-\pi_{3}(\be)\geq n-k+1$. Using the second expression of $A_{n}\be$, we obtain $\pi_{3}(\be)-\pi_{1}(\be)\geq k$. Suppose that $\pi_{1}(\be)-\pi_{3}(\be)\geq n-k+1$. Then

\begin{align*}
    f_{k}(A_{n}\be)= & k(\pi_{1}(\be)-\pi_{3}(\be))+\pi_{2}(\be)+(n+1)\pi_{3}(\be)\\
    \geq & k(n-k+1)+\pi_{2}(\be)+(n+1)\pi_{3}(\be)\\
    \geq & k(n-k+1). 
\end{align*}
Now suppose that  $\pi_{3}(\be)-\pi_{1}(\be)\geq k$. In particular, $\pi_{3}(\be)\geq k$. Then 
\begin{align*}
       f_{k}(A_{n}\be)= & k(\pi_{1}(\be)-\pi_{3}(\be))+\pi_{2}(\be)+(n+1)\pi_{3}(\be)\\
    = & (n+1)\pi_{3}(\be)-k\pi_{3}(\be)+k\pi_{1}(\be)+\pi_{2}(\be)\\
    = & (n-k+1)\pi_{3}(\be)+k\pi_{1}(\be)+\pi_{2}(\be)\\
    \geq & (n-k+1)k+k\pi_{1}(\be)+\pi_{2}(\be)\\
    \geq & nk-k^{2}+k.
\end{align*}
In any case, we obtain that $f_{k}(A_{n}\be)\geq k(n-k+1)$ for all $\be\in\N^{3}$ with the conditions of the lemma as we claim.
\end{proof}

\begin{lem}\label{lema33}Let $v=v_{l,\eta_{k}}+q(1,1)\in\N^{2}$ be with $l\leq n$ and $q\geq n-l+1+r_{l,\eta_{k}}$. Then 
$f_{k}(v)\geq f_{k}(u)$ for all $u\in T_{0,\eta_{k}}\bigcup\cup_{j=1}^{l}T_{j,\eta_{k}}+r_{j,\eta_{k}}$.

\end{lem}
\begin{proof}
We proceed by induction over $l$. Consider $l=1$. Then $v=v_{1,\eta_{k}}+q(1,1)$, with $q\geq n+r_{1,\eta_{k}}$ and we need to prove that $f_{k}(v)\geq f_{k}(u)$ for all $u\in T_{0,\eta_{k}}\cup T_{1,\eta_{k}}+r_{1,\eta_{k}}$. 
If $u\in T_{1,\eta_{k}}+r_{1,\eta_{k}}$, then $u=v_{1,\eta_{k}}+(p+r_{1,\eta_{k}})(1,1)$, with $p\leq n-1$. It follows that 
\begin{align*}
    f_{k}(v)= & f_{k}(v_{1,\eta_{k}}+q(1,1))\\
    = & f_{k}(v_{1,\eta_{k}})+q\\
    \geq & f_{k}(v_{1,\eta_{k}})+n+r_{1,\eta_{k}}\\
    \geq &  f_{k}(v_{1,\eta_{k}})+p+r_{1,\eta_{k}}\\
    = & f_{k}(v_{1,\eta_{k}}+p+r_{1,\eta_{k}}(1,1))\\
    = & f_{k}(u).
\end{align*}
If $u\in T_{0,\eta_{k}}$, then $u=p(1,1)$, with $p\leq n$. By definition \ref{Teta}, $v_{1,\eta_{k}}=(1,0)$ or $v_{1,\eta_{k}}=(0,1)$. Thus $f_{k}(v)=k+q\geq k+n$ or $f_{k}(v)=(1-k)+q\geq n+(n-k+1)$. Since $k\in\{1,\ldots,n\}$, in any case we have that 
\begin{equation}\label{eq1}
    f_{k}(v)>n\geq p=f_{k}(u).
\end{equation}
We conclude that is true for $l=1$. 

Now, suppose that is true for all $l'<l$, i.e., $f_{k}(v_{l',\eta_{k}}+q'(1,1))\geq f_{k}(u)$ for all $u\in T_{0,\eta_{k}}\bigcup\cup_{j=1}^{l'}T_{j,\eta_{k}}+r_{j,\eta_{k}}$ and $q'\geq n-l'+1+r_{l',\eta_{k}}$. Let $v=v_{l,\eta_{k}}+q(1,1)$ with $q\geq n-l+1+r_{l,\eta_{k}}$. If $u\in T_{l,\eta_{k}}+r_{l,\eta_{k}}$, we have that $u=v_{l,\eta_{k}}+(p+r_{l,\eta_{k}})(1,1)$, with $p\leq n-l$. Then
\begin{align*}
    f_{k}(v)= & f_{k}(v_{l,\eta_{k}})+q\\
    \geq & f_{k}(v_{l,\eta_{k}})+n-l+1+r_{l,\eta_{k}}\\
    \geq & f_{k}(v_{l,\eta_{k}})+p+r_{l,\eta_{k}}\\
    = & f_{k}(v_{l,\eta_{k}}+(p+r_{l,\eta_{k}})(1,1))\\
    = & f_{k}(u).
\end{align*}
Thus $f_{k}(v)\geq f_{k}(u)$ for all $u\in T_{l,\eta_{k}}+r_{l,\eta_{k}}$. Consider $u\in T_{l-1,\eta_{k}}+r_{\eta_{k},l-1}$. By definition, $u=v_{l-1,\eta_{k}}+(p+r_{l-1,\eta_{k}})(1,1)$ with $p\leq n-l+1$. By lemma \ref{propetak} $4)$, $f_{k}(v_{l,\eta_{k}}+r_{l,\eta_{k}})\geq f_{k}(v_{l-1,\eta_{k}}+r_{l-1,\eta_{k}})$. Then 
\begin{align*}
    f_{k}(v)= & f_{k}(v_{l,\eta_{k}})+q\\
    \geq & f_{k}(v_{l,\eta_{k}})+n-l+1+r_{l,\eta_{k}}\\
    = & f_{k}(v_{l,\eta_{k}}+r_{l,\eta_{k}}(1,1))+n-l+1\\
    \geq & f_{k}(v_{l-1,\eta_{k}}+r_{l-1,\eta_{k}}(1,1))+p\\
    = & f_{k}(v_{l-1,\eta_{k}}+(p+r_{l-1,\eta_{k}})(1,1))\\
    = & f_{k}(u).
\end{align*}
Obtaining that the statement is true for all $u\in T_{l-1,\eta_{k}}+r_{l-1,\eta_{k}}$.

Suppose  $f_{k}(v_{l,\eta_{k}}+r_{l,\eta_{k}})\geq f_{k}(v_{l-1,\eta_{k}}+r_{l-1,\eta_{k}})+1$. Obtaining that
\begin{align*}
    f_{k}(v)\geq & f_{k}(v_{l,\eta_{k}}+r_{l,\eta_{k}}(1,1))+n-l+1\\
    \geq & f_{k}(v_{l-1,\eta_{k}}+r_{l-1,\eta_{k}}(1,1))+n-l+2.
\end{align*}
Then, by the induction hypothesis over $l-1$, $f_{k}(v)\geq f_{k}(u)$ for all $u\in T_{0,\eta_{k}}\bigcup\cup_{j=1}^{l-1}T_{j,\eta_{k}}+r_{j,\eta_{k}}$, obtaining the result.

Now suppose that $f_{k}(v_{l,\eta_{k}}+r_{l,\eta_{k}})=f_{k}(v_{l-1,\eta_{k}}+r_{l-1,\eta_{k}})$. For this, we have two cases, if $l=2$ or $l>2$. If $l=2$, by (\ref{eq1}), we have that 
\begin{align*}
    f_{k}(v)= & f_{k}(v_{2,\eta_{k}})+q \\
    \geq & f_{k}(v_{2,\eta_{k}})+n-1+r_{2,\eta_{k}}\\
    = & f_{k}(v_{2,\eta_{k}}+r_{2,\eta_{k}}(1,1))+n-1\\
    = & f_{k}(v_{1,\eta_{k}}+r_{1,\eta_{k}}(1,1))+n-1\\
    = & f_{k}(v_{1,\eta_{k}}+(n-1+r_{1,\eta_{k}})(1,1))\\
    \geq & n\\
    \geq & f_{k}(u),
\end{align*}
for all $u\in T_{0,\eta_{k}}$. If $l>2$, then for all $u\in T_{0,\eta_{k}}\bigcup\cup_{j=1}^{l-2}T_{j,\eta_{k}}+r_{j,\eta_{k}}$, we have that
\begin{align*}
    f_{k}(v)\geq & f_{k}(v_{l,\eta_{k}}+r_{l,\eta_{k}}(1,1))+n-l+1\\
    \geq & f_{k}(v_{l-2,\eta_{k}}+r_{l-2,\eta_{k}}(1,1))+n-l+3\\
    \geq & f_{k}(u),
\end{align*}
where the second inequality comes from $5)$ of lemma \ref{propetak} and the last inequality comes from the induction hypothesis over $l-2$. Obtaining the result.

\end{proof}

Now we are ready to prove the other important result of this section.

\begin{pro}\label{final2}
Let $\eta_{k}\in\Omega$ and its respective $J_{\eta_{k}}\in S_{A_{n}}$. Then for all $J\in S_{A_{n}}$ we have that $f_{k}(m_{J_{\eta_{k}}})\leq f_{k}(m_{J})$.
\end{pro}

\begin{proof}
Let $J=\{\beta_{1},\ldots,\beta_{\lambda_{2,n}}\}\in S_{A_{n}}$. By definition of $S_{A_{n}}$ we have that $0\neq \det\Big(c_{\beta_{i}}\Big)_{1\leq i\leq\lambda_{2,n}}$, where $c_{\beta_{i}}:=\sum_{\gamma\leq\beta_{i}}
(-1)^{|\beta_{i}-\gamma|}\binom{\beta_{i}}{\gamma}\overline{A_{n}\ga}$ (Recall notation \ref{notacion}). 

Fixing the $\be_{1}$th row of this matrix and using basic properties of determinants, we obtain that

$$0\neq \det\Big(c_{\beta_{i}}\Big)_{1\leq i\leq\lambda_{2,n}}=\sum_{\ga\leq\beta_{1}}(-1)^{|\beta_{i}-\gamma|}\binom{\beta_{i}}{\gamma}\det\left( \begin{array}{c}
\overline{A_{n}\ga} \\
c_{\be_{2}} \\
\ldots \\
c_{\beta_{\lambda_{2,n}}}\end{array} \right).$$
Since the determinant is not zero, this implies that there exists $\be_{1}'\leq\be_{1}$ such that $$\det\left( \begin{array}{c}
\overline{A_{n}\be_{1}'} \\
c_{\be_{2}} \\
\ldots \\
c_{\beta_{\lambda_{2,n}}}\end{array} \right)\neq0.$$

Applying this process for each row, we obtain the a of vectors $B=\{\be_{i}'\}_{i=1}^{\lambda_{2,n}}\subset\Lambda_{3,n}$ such that $\be_{i}'\leq\be_{i}$ for all $i\in\{1,\ldots,\lambda_{2,n}\}$ and with the property $\det \Big(\overline{A_{n}\be_{i}'}\Big)_{1\leq i\leq\lambda_{2,n}}\neq0$.

The goal is to construct a bijective correspondence $\varphi:B\rightarrow T_{\eta_{k}}'$, such that $f_{k}(A_{n}\be_{i}')\geq f_{k}(\varphi(\be_{i}'))$. Consider the set $v_{j,\eta_{k}}+L:=\{v_{j\eta_{k}}+p(1,1)\mid p\in\N\}$. Now, consider the following partition of $B$:
$$B_{0}=\{\be'_{i}\in B\mid A_{n}\be'_{i}\in T_{\eta_{k}}'\},$$
$$B_{1}=\{\be'_{i}\in B\mid A_{n}\be'_{i}\in (v_{j,\eta_{k}}+L)\setminus T_{j,\eta_{k}}+r_{j,\eta_{k}}\,\,\,\mbox{for some}\,\,\,j\in\{1,\dots,n\}\},$$
$$B_{2}=\{\be'_{i}\in B\mid A_{n}\be_{i}'=q(1,1) \,\,\mbox{for some}\,\, q> n\},$$
$$B_{3}=\{\be'_{i}\in B\mid A_{n}\be'_{i}\notin (v_{j,\eta_{k}}+L)\,\,\,\mbox{for all}\,\,\,j\in\{0,\dots,n\}\},$$

    For all $\be_{i}'\in B_{0}$, we define $\varphi(\be_{i}')=A_{n}\be_{i}'$. Since $\det \Big(\overline{A_{n}\be_{i}'}\Big)_{1\leq i\leq\lambda_{2,n}}\neq0$, we have that $\varphi(\be_{i}')\neq\varphi(\be_{j}')$ for all $\be_{i}',\be_{j}'\in B_{0}$. 
    
    Now, if $B_{1}\neq\emptyset$, we rearrange $B$ in such a way that $\{\be_{1}',\be_{2}',\ldots,\be_{m}'\}=B_{1}$. Consider $\be_{1}'\in B_{1}$. By construction of $B_{1}$, there exist $l\leq n$ and $q\in\N$ such that $A_{n}\be_{1}'=v_{l,\eta_{k}}+q(1,1)$. By  proposition \ref{gene} we have that 
    \begin{multline*}
        \overline{A_{n}\be_{1}'}\in \spn_{\C}\{\overline{v}\in\C^{\lambda_{2,n}}\mid v\in \bigcup_{j=0}^{l}T_{j,\eta_{k}}\}\\
    =\spn_{\C} \{\overline{v}\in\C^{\lambda_{2,n}}\mid v\in T_{0,\eta_{k}}\bigcup\cup_{j=1}^{l}T_{j,\eta_{k}}+r_{j,\eta_{k}}\}.
        \end{multline*}
This implies that $\overline{A_{n}\be_{1}'}=\sum_{v\in T_{0,\eta_{k}}\bigcup\cup_{j=1}^{l}T_{j,\eta_{k}}+r_{j,\eta_{k}}}a_{v}\overline{v}$, for some constants $a_{v}\in\C$. Using again basic properties of the determinant, we obtain there exists $u_{\be_{1}'}\in T_{0,\eta_{k}}\bigcup\cup_{j=1}^{l}T_{j,\eta_{k}}+r_{j,\eta_{k}}$ such that 
$$\det\left( \begin{array}{c}
\overline{u_{\be_{1}'}} \\
\overline{A_{n}\be_{2}'} \\
\vdots \\
\overline{A_{n}\beta_{\lambda_{2,n}}'}\end{array} \right)\neq 0.$$
Applying this process for each element of $B_{1}$, we obtain the vectors $\{u_{\be_{j}'}\}_{j=1}^{m}$. We define $\varphi(\be_{j}')=u_{\be_{j}'}$ for all $j\in\{1,\ldots,m\}$. Now, we need to check that $\varphi$ is injective on $B_{0}\cup B_{1}$ and $f_{k}(A_{n}\be_{i}')\geq f_{k}(\varphi(\be_{i}'))$.

Notice that, by construction
$$\det\left( \begin{array}{c}
\overline{u_{\be_{1}'}} \\
\vdots \\
\overline{u_{\be_{m}'}}\\
\overline{A_{n}\be_{m+1}'} \\
\vdots \\
\overline{A_{n}\beta_{\lambda_{2,n}}'}\end{array} \right)\neq 0.$$
This implies that $\overline{u_{\be_{i}'}}\neq\overline{u_{\be_{j}'}}$ for all $1\leq i<j\leq m$. In particular we have that $u_{\be_{i}'}\neq u_{\be_{j}'}$. Moreover, using the same argument, we have that $u_{\be_{i}'}\neq A_{n}\be_{j}'=\varphi(\be_{j}')$ for all $\be_{j}'\in B_{0}$. Thus, $\varphi$ is injective on $B_{0}\cup B_{1}$. 

On the other hand, $A_{n}\be_{i}'=v_{l,\eta_{k}}+q(1,1)\notin T_{l,\eta_{k}}+r_{l,\eta_{k}}$ for some $l\leq n$. This implies that $q\geq n-l+1+r_{l,\eta_{k}}$. Then
\begin{align*}
    f_{k}(A_{n}\be_{i}')= & f_{k}(v_{l,\eta_{k}}+q(1,1)) \\
    = & f_{k}(v_{l,\eta_{k}})+q\\
    \geq & f_{k}(v_{l,\eta_{k}})+n-l+1+r_{l,\eta_{k}}\\
    = & f_{k}(v_{l,\eta_{k}}+(n-l+1+r_{l,\eta_{k}})(1,1))\\
    \geq & f_{k}(u_{\be_{i}'}).
\end{align*}
where the last inequality comes from lemma \ref{lema33}, obtaining the inequality we are looking for.

For all $\be_{i}'\in B_{2}$, we have that $A_{n}\be_{i}'=q(1,1)$, for some $q>n$. By lemma \ref{depenrecta}, we have that $\overline{A_{n}\be_{i}'}=\sum_{v\in T_{0,\eta_{k}}}a_{v}\overline{v}$. Applying the same method that for the elements of $B_{1}$, we define can  $\varphi$ with the properties that we need.

Now, since $\mid T_{\eta_{k}}'\mid=\mid B\mid=\lambda_{2,n}$, we have that $\mid B_{3}\mid=\mid T_{\eta_{k}}\setminus\{\varphi(\be_{j}')\mid\be_{j}'\in B_{0}\cup B_{1}\cup B_{2}\}\mid$. Then we take $\varphi(\be_{i}')=v$, with $v\in T_{\eta_{k}}\setminus\{\varphi(\be_{j}')\mid\be_{j}'\in B_{0}\cup B_{1}\cup B_{2}\}$ in such a way that $\varphi(\be_{i}')\neq\varphi(\be_{j}')$ for all $\be_{i}',\be_{j}'\in B_{3}$ and $\be_{i}'\neq\be_{j}'$. 

By construction we obtain that $\varphi$ is a bijective correspondence and by definition of $B_{3}$ and lemma \ref{lema32} we have that $f_{k}(A_{n}\be_{i}')\geq f_{k}(\varphi(\be_{i}'))$ for all $\be_{i}'\in B_{3}$. Then 
\begin{align*}
   f_{k}(m_{J})= & \sum_{\be_{i}\in J}f_{k}(A_{n}\be_{i}) \\
    \geq & \sum_{\be_{i}'\in B}f_{k}(A_{n}\be_{i}')\\
    \geq & \sum_{\be_{i}'\in B}f_{k}(\varphi(\be_{i}'))\\
    = & \sum_{b\in T_{\eta_{k}}'}f_{k}(v)\\
      = & f_{k}(m_{J_{\eta_{k}}}),
\end{align*}
where the first inequality comes from lemma \ref{lema31} and the second comes from the construction of $\varphi$.
\end{proof}

Now we are ready to prove Theorem \ref{mainn}.

\begin{proof}
 By proposition \ref{corresomegasa}, $J_{\eta_{k}}\in S_{A_{n}}$. By corollary \ref{final1}, there exists $J_{\eta}\in S_{A_{n}}$ such that $J_{\eta}\neq J_{\eta_{k}}$ and $f_{k}(m_{J_{\eta}})=f_{k}(m_{J_{\eta_{k}}})$. Using  proposition \ref{final2}, we obtain that $\mbox{ord}_{I_{n}}((k,1-k))=f_{k}(m_{J_{\eta_{k}}})$. This implies that $(k,1-k)\in \sigma_{m_{J_{\eta}}}\cap\sigma_{m_{J_{\eta_{k}}}}$.

\end{proof}

\end{subsection}

\section*{Acknowledgements}
I would like to thank Daniel Duarte for follow my work and stimulating discussions. I also thank Takehiko Yasuda for propose me this problem.

\end{section}

\vspace{.5cm}

\noindent{\footnotesize {Enrique Ch\'avez-Mart\'inez,\\
Universidad Nacional Aut\'onoma de M\'exico.\\
Av. Universidad s/n. Col. Lomas de Chamilpa, C.P. 62210, Cuernavaca, Morelos, Mexico.\\
Email: enrique.chavez@im.unam.mx}

\end{document}